\pdfoutput=1

\documentclass[10pt,a4paper]{scrartcl}
\usepackage[utf8]{inputenc}
\usepackage[T1]{fontenc}

\usepackage{hyperref}

\hypersetup%
{%
	pdftex,
	colorlinks,
	pdfpagelabels,
	urlcolor=red,
	linkcolor=teal,
	citecolor=orange,
	unicode,
	hyperindex=true,
	final=true,
	pdfauthor=Alexander Körschgen,
	pdftitle=A Comparison of two Models of Orbispaces,
	pdfsubject={2010 AMS MSC: 55R91 (Primary), 18D20},
	pdfkeywords=
	{%
		Orbispace,
		Unstable Global Homotopy Theory,
		Compact Lie Group,
		Homotopy Quotient
	}
}

\usepackage{mystyle-minimal-koma}
\usepackage{mystyle-operators}
\usepackage{mystyle-environments-numbered}
\usepackage{mydiagrams}

\usepackage[nottoc]{tocbibind}

\usepackage{nicefrac}
\usepackage{stmaryrd} 
\usepackage{mathtools} 
\usepackage{tensor}

\usepackage{pdfsync}

\usepackage{bookmark}

\newcommand{\ri}{\mathbb{R}^\infty}
\newcommand{\orbs}{\mathbf{O}_{\operatorname{gl}}}
\newcommand{\orbsPDF}{O\_gl}

\renewcommand{\L}{\mathbb{L}}
\newcommand{\cL}{\mathcal{L}}

\newcommand{\orbgh}{\operatorname{Orb}}
\newcommand{\orbghPDF}{Orb}

\newcommand{\orbh}{\orbgh^\prime}
\newcommand{\orbhPDF}{Orb′}
\newcommand{\etPDF}{Ẽ} 

\newcommand{\mono}{{\operatorname{mono}}}
\newcommand{\orbhmono}{\orbgh^{\prime, \mono}}

\DeclareMathOperator{\Rep}{Rep}
\DeclareMathOperator{\graph}{graph}
\newcommand{\cF}{\mathcal{F}}

\newcommand{\et}{\widetilde{E}}

\DeclareMathOperator{\Stab}{Stab}

\newcommand{\obo}{\mathbb{O}}
\newcommand{\catc}{\mathcal{C}}
\newcommand{\catd}{\mathcal{D}}

\newcommand{\cG}{\mathcal{G}}
\newcommand{\grpd}{{\operatorname{Grpd}}}
\newcommand{\grpdtop}{\grpd_{\Top}}
\newcommand{\B}{\mathbb{B}}
\newcommand{\maptopgrpd}{\map_{\grpdtop}}

\DeclareMathOperator{\pr}{pr}
\DeclareMathOperator{\fib}{fib}

\newcommand{\R}{\mathbb{R}}

\DeclareMathOperator{\sk}{sk}
\newcommand{\del}{\partial}

\DeclareMathOperator{\sti}{Stiefel}
\newcommand{\cV}{\mathcal{V}}
\newcommand{\cW}{\mathcal{W}}

\deffootnote[1em]{1em}{1em}{\textsuperscript{\thefootnotemark}}

\title%
{%
	A Comparison of two Models of Orbispaces\footnote%
	{%
		Published in \emph{Homology, Homotopy and Applications}, vol.~20(1),
		2018, pp.329--358,\newline
		available at
		\href{http://dx.doi.org/10.4310/HHA.2018.v20.n1.a19}
		{doi:10.4310/HHA.2018.v20.n1.a19}.
	}
}

\author%
{%
	Alexander Körschgen\footnote%
	{%
		The author was supported in part by a grant from the International Max
		Planck Research School on Moduli Spaces.
	}
}
\date{}
\begin{document}%
\maketitle%

\begin{abstract}
This paper proves that the two homotopy theories for or\-bi\-spaces
given by Gepner and Henriques and by Schwede, respectively, agree by
providing a zig-zag of Dwyer-Kan equivalences between the respective
topologically enriched index categories. The aforementioned authors establish 
various models for unstable global homotopy theory with compact Lie group
isotropy, and orbispaces serve as a common denominator for their particular
approaches. Although the two flavors of orbispaces are expected to agree with
each other, a concrete comparison zig-zag has not been known so far. We bridge
this gap by providing such a zig-zag which asserts that all those models for
unstable global homotopy theory with compact Lie group isotropy which have
been described by the authors named above agree with each other.

On our way, we provide a result which is of independent interest. For a large
class of free actions of a compact Lie group, we prove that the homotopy
quotient by the group action is weakly equivalent to the strict quotient. This
is a known result under more restrictive conditions, e.g., for free actions on
a manifold. We broadly extend these results to all free actions of a compact
Lie group on a compactly generated Hausdorff space.
\end{abstract}

\pdfbookmark[section]{\contentsname}{toc}
\tableofcontents

\section{Introduction}
Global homotopy theory is concerned with \emph{global spaces} which can be
thought of as spaces that are simultaneously equipped with compatible actions
of all compact Lie groups.
For any such group $G$, a global space admits a notion of $G$-fixed points,
and a map between two global spaces is a \emph{global equivalence} if it
induces a weak equivalence on all of these fixed point spaces.
There are also versions for other classes of groups which we elaborate on in
Subsection~\ref{subsec:relative}.

In order to make the notion of global spaces more precise, various models
have emerged:
In \cite{gepner_homotopy_2007}, Gepner and Henriques 
equip stacks and topological groupoids with certain homotopical structures and
prove that they yield homotopy theories that are equivalent to their version
of orbispaces. This gives rise to three different models for global spaces.

In \cite{schwede_orbispaces_2017}, Schwede introduces $\cL$-spaces and
orthogonal spaces along with his version of orbispaces, together with model
category structures on all of these categories, and verifies that they are
Quillen equivalent. Therefore, he also describes three models for global
spaces.

In both cases, orbispaces are space-valued enriched presheaves on a
topologically enriched small category, the \emph{orbit category}, equipped
with the projective model structure.  However, the two variants of orbispaces
mentioned above are not equal on the nose because the
respective orbit categories differ.

The orbit category used in~\cite{gepner_homotopy_2007} is called $\orbgh$, and
the one from~\cite{schwede_orbispaces_2017} is denoted by $\orbs$. 
As we will recall, there is significant evidence that the two versions of
orbispaces should be equivalent. The objects of both $\orbgh$ and $\orbs$ are
given by compact Lie groups.
Given two such groups $H,G$, the mapping spaces $\orbs(H,G)$ and
$\orbgh(H,G)$ have the same weak
homotopy type, see Remark~\ref{rem:orbshgmisc}, but, at least to the author's
knowledge, no concrete zig-zag of maps has been written down so far.

This exposition shows that there is a zig-zag of weak equivalences between the
respective mapping spaces which is compatible with the given structures of
topologically enriched categories. We deduce that there is a zig-zag of
Dwyer-Kan equivalences between the orbit category $\orbs$
from~\cite{schwede_orbispaces_2017} and the orbit category $\orbgh$
from~\cite{gepner_homotopy_2007}. This
implies by~\cite{gepner_homotopy_2007,koerschgen_dwyer-kan_2017} that the
associated presheaf categories are Quillen equivalent via a zig-zag, and as a
consequence, all the models for global homotopy theory from both papers are
equivalent to each other.


\subsection{Results}
For two universal subgroups (see Definition~\ref{def:orbsgeneral})
$G,H \subseteq \L(\ri,\ri)$, set
\[
	\orbs(H,G) := (\mathbb{L}(\ri_G, \ri_H) /G)^H
\]
where $\L$ denotes the space of isometric linear embeddings and $\ri_G, \ri_H$
are both just $\ri$ equipped with the canonical $G$- and $H$-action respectively.
Moreover, let
\[
	\orbgh(H,G) := \map(H,G) \times_G EG
\]
where $\map(H,G)$ denotes the space of continuous (and hence smooth) group
homomorphisms from $H$ to $G$.

These are the mapping spaces of the topologically enriched categories $\orbs$
and $\orbgh$ which share the same set of objects, namely $\ob \orbs = \ob
\orbgh$ is the set of all universal subgroups of $\L(\ri,\ri)$. Note that a
topologically enriched functor which is the identity on objects is a Dwyer-Kan
equivalence if and only if it is a weak equivalence on all mapping spaces
(see Definition~\ref{def:dk} for the generic definition of a Dwyer-Kan
equivalence).

\begin{thm*}
	There is a zig-zag of two Dwyer-Kan equivalences, both of which are the
	identity on objects, between $\orbs$ and $\orbgh$.
\end{thm*}
This appears as Corollary~\ref{cor:zigzagdk}.

\begin{cor*}
There is a zig-zag of Quillen equivalences between the projective model
structures on the enriched presheaf categories $\Pre(\orbs,\Top)$ and
$\Pre(\orbgh,\Top)$.
\end{cor*}
This follows immediately using~\cite[Lemma~A.6]{gepner_homotopy_2007} or the
in-depth account~\cite[Theorem~3.5]{koerschgen_dwyer-kan_2017}.
Here, the first model category, $\Pre(\orbs,\Top)$ is the variant of orbispaces used
in~\cite{schwede_orbispaces_2017} while the second one, $\Pre(\orbgh,\Top)$ is
used in~\cite{gepner_homotopy_2007}.

We will discuss generalizations of these results for a monomorphism variant of
the orbit category and for global family versions in Section~\ref{sec:results}.

\subsection{Organization of the Paper}
The promised zig-zag of Dwyer-Kan equivalences between $\orbs$ and $\orbgh$
will consist of a third topologically enriched category $\orbh$ with $\ob
\orbh = \ob \orbs = \ob \orbgh$.

Section~\ref{sec:intermmappingspace} focuses
on the mapping space $\orbh(H,G)$ for two fixed universal subgroups
$H,G \subseteq \L(\ri,\ri)$ and its relation to the mapping spaces $\orbs(H,G)$ and
$\orbgh(H,G)$. While Subsection~\ref{subsec:orbghrem} recalls the definition
of $\orbgh(H,G)$ along with auxiliary results, Subsection~\ref{subsec:orbsrem}
proceeds similarly for $\orbs(H,G)$.

In Subsection~\ref{subsec:intermediatemappingspace}, we construct a free
$G$-space $\et(H,G)$, see Definition~\ref{def:ethg}, and set $\orbh(H,G) :=
\et(H,G) \times_G EG$. This space admits a canonical map to $
\orbgh(H,G)$, which is readily seen to be a weak equivalence, see
Proposition~\ref{prop:lhgwe}.

Therefore, it remains to specify a weak equivalence $\orbh(H,G) \to
\orbs(H,G)$, which is the content of Subsection~\ref{subsec:orbhtoorbs}. We
show that the quotient space $\et(H,G)/G$ is canonically homeomorphic to
$\orbs(H,G)$, see Proposition~\ref{prop:orbstriangle}. After showing that the
$G$-action on $\et(H,G)$ is free, we deduce that the natural map $\orbh(H,G) =
\et(H,G) \times_G EG \to \et(H,G) / G \cong \orbs(H,G)$ is a weak equivalence
(Proposition~\ref{prop:khgwe}). This concludes the construction of the zig-zag
on the level of individual mapping spaces:
\[
\begin{tikzcd}
	\orbgh(H,G) &
	\et(H,G) \times_G EG = \orbh(H,G) \lar[swap]{\simeq} \rar{\simeq} &
	\et(H,G) / G \cong \orbs(H,G).
\end{tikzcd}
\]

The goal of Section~\ref{sec:zigzagdk} is the extension of this zig-zag to a
zig-zag of Dwyer-Kan equivalences. However, the construction of the
topologically enriched category $\orbh(H,G)$ is not complete since we have not
specified composition laws yet. This is amended in
Subsection~\ref{subsec:comporbhorbgh}: As the composition law on $\orbh$ is
similar to the one on $\orbgh$, we simultaneously recall the
composition on $\orbgh$ and define the one on $\orbh$. In the end, we conclude
that the maps $\orbh(H,G) \to \orbgh(H,G)$ from
Subsection~\ref{subsec:intermediatemappingspace} extend to a Dwyer-Kan
equivalence $\orbh \to \orbgh$.
Finally, it is left to verify that the composition law on $\orbs$ is
compatible with the maps $\orbh(H,G) \to \orbs(H,G)$ introduced in
Subsection~\ref{subsec:orbhtoorbs}, which is the content of
Subsection~\ref{subsec:comporbhorbs}. In conclusion, we deduce that there is a
zig-zag of Dwyer-Kan equivalences:
\[
\begin{tikzcd}
	\orbgh &
	\orbh \lar[swap]{\simeq} \rar{\simeq} &
	\orbs.
\end{tikzcd}
\]

From the various auxiliary results of the appendix, we would like to highlight
a particular one: Theorem~\ref{thm:freeactionT2space} proves that whenever a compact
Lie group $G$ acts freely on a Hausdorff space $X$, the canonical map
\[
	X \times_G EG \to X/G
\]
is a weak equivalence.

\subsection*{Acknowledgments}
The author would like to acknowledge that the main idea for the construction
of the zig-zag of Dwyer-Kan equivalences is due to Thomas Nikolaus. He also
provided helpful hints and remarks concerning proof strategies and technical
details.

Moreover, the author would like to thank Fabian Hebestreit, Irakli
Patchkoria, and Christian Wimmer for fruitful discussions concerning the fiber
sequences in Subsection~\ref{subsec:freeactionhomquot}.
Finally, the author is grateful for advice and guidance by Stefan Schwede and
Peter Teichner.

\section{The Intermediate Mapping Space}
\label{sec:intermmappingspace}
Throughout the main body of this paper, the term \emph{space} will always
refer to a compactly generated weak Hausdorff space. This convention is
explained in more detail in Subsection~\ref{subsec:cgwhspaces} of the
appendix.

\subsection{A Reminder on \texorpdfstring{$\orbgh(H,G)$}{\orbghPDF(H,G)}}
\label{subsec:orbghrem}
\begin{definition}
	Let $H,G$ be Lie groups.
	\begin{enumerate}
		\item Denote by $\map(H,G)$ the space of all continuous (hence smooth)
			group homomorphisms. It comes with a continuous $G$-action from the
			right by conjugation, i.e., for $\alpha\colon H \to G$ and $g \in G$,
			define \mbox{$(\alpha \cdot g)(h) := g^{-1} \alpha(h) g$}.
		\item The quotient of $\map(H,G)$ by this (usually non-free) $G$-action
			is $\Rep(H,G)$.
		\item Using the bar construction as a model for $EG$, we set
			\[
				\orbgh(H,G) := \map(H,G) \times_G EG := 
					|n \mapsto \map(H,G) \times G^n |
			\]
			as the realization of a simplicial topological space.
	\end{enumerate}
\end{definition}

\begin{rem}
	Gepner and Henriques define $\orbgh(H,G)$ as the \emph{fat} geometric
	realization of the simplicial topological space $\{n \mapsto \map(H,G)
	\times G^n\}$, which is weakly equivalent to our definition.
	We prefer the non-fat realization as it behaves better from a technical
	point of view for our purposes because it commutes with products on the
	nose. We will elaborate on the consequences in
	Remark~\ref{rem:orbghfatunclear}.
\end{rem}

In preparation for later results, we derive an important property of
the topology on $\map(H,G)$.

\begin{lemma}
	[``Close Maps are Conjugate'',
		{\cite[Lemma 38.1]{conner_differentiable_1964}}]
	Let $H$ be a compact Lie group and $G$ a Lie group. For each 
	$\alpha \in \map(H,G)$, there is an open neighborhood $U \subseteq H \times G$
	of $\graph(\alpha)$ such that for each $\beta \in \map(H,G)$ with
	$\graph(\beta) \subseteq U$, $\beta$ is conjugate to $\alpha$.
\end{lemma}

\begin{prop}~
	\label{prop:maphgdecomp}
	\begin{enumerate}
		\item
			Let $H$ be a compact Lie group and $G$ a Lie group. For each $\alpha \in
			\map(H,G)$, there is an open neighborhood $V$ of $\alpha$ in $\map(H,G)$
			such that each $\beta \in V$ is conjugate to $\alpha$.
		\item	
			The quotient topology on $\Rep(H,G) = \map(H,G)/G$ is
			discrete.
		\item
			All $G$-orbits $\alpha G$ are open.  
			In particular, there is a $G$-equivariant decomposition
			\[
				\map(H,G) \cong \coprod_{i \in I} \alpha_i G
			\]
			where $\{\alpha_i\}_{i \in I}$ is a choice of representatives
			of equivalence classes in $\Rep(H,G)$.
	\end{enumerate}
\end{prop}
\begin{proof}
	\begin{enumerate}
		\item Pick a metric $d_G$ on $G$. As $H$ is compact, the topology on
			$\map(H,G)$ coincides with the one induced by the metric $d$ which
			is defined as follows:
			\[
				d(\alpha,\beta) := \max_{h \in H} d_G(\alpha(h),\beta(h)).
			\]

			Fix $\alpha \in \map(H,G)$ and choose an open $U$ as in the
			previous Lemma. For each $h \in H$, pick an open neighborhood $U_h
			= A_h \times B_{\epsilon(h)} (\alpha(h))$ of $(h,\alpha(h))$
			inside $U$ where $B_\epsilon(\alpha(h))$ denotes an
			$\epsilon$-ball around $\alpha(h)$ with respect to $d_G$. By
			replacing $A_h$ with $A_h \cap \alpha^{-1}
			(B_{\nicefrac{\epsilon(h)}{2}}(\alpha(h)))$, we may assume without loss of
			generality that $\alpha(A_h) \subseteq
			B_{\nicefrac{\epsilon(h)}{2}}(\alpha(h))$.

			The $A_h$ cover the compact space $H$. Therefore, we find a finite
			set $H_0 \subseteq H$ such that $\{A_h\sth h \in H_0\}$ covers
			$H$. Set $\epsilon := \min_{h \in H_0} \epsilon(h)$ and denote by
			\[
				V = B_{\nicefrac{\epsilon}{2}}(\alpha) = \{ \beta \in \map(H,G)\sth
					d(\alpha,\beta) = \max_{h \in H} d_G (\alpha(h),\beta(h))
					< \nicefrac{\epsilon}{2} \}
			\]
			the $\nicefrac{\epsilon}{2}$-ball around $\alpha$ in $\map(H,G)$.

			Pick $\beta \in V$. We claim that $\graph(\beta) \subseteq U$. Let
			$h \in H$, then $h \in A_{h^\prime}$ for some $h^\prime \in H_0$.
			Also, $d_G (\alpha(h),\beta(h)) < \nicefrac{\epsilon}{2} \leq
			\nicefrac{\epsilon(h^\prime)}{2}$
			because $\beta \in V$. As $\alpha(A_{h^\prime})$ is contained in 
			$B_{\nicefrac{\epsilon(h^\prime)}{2}} (\alpha(h^\prime))$, we obtain
			$d_G(\alpha(h^\prime),\alpha(h)) <
			\nicefrac{\epsilon(h^\prime)}{2}$. By the triangle inequality,
			\[
				d_G (\alpha(h^\prime),\beta(h)) \leq
				d_G (\alpha(h^\prime),\alpha(h)) + d_G (\alpha(h),\beta(h))
				< \nicefrac{\epsilon(h^\prime)}{2} +
				\nicefrac{\epsilon(h^\prime)}{2} = \epsilon(h^\prime).
			\]
			Therefore, $(h,\beta(h)) \in
			A_{h^\prime} \times B_{\epsilon(h^\prime)} (\alpha(h^\prime))
			\subseteq U$, and $\graph(\beta) \subseteq U$. This  
			implies that $\beta$ is conjugate to $\alpha$ by the choice of
			$U$.
		\item The quotient projection $\map(H,G) \to \Rep(H,G) = \map(H,G) /
			G$ is open. For $\alpha \in \map(H,G)$, an open neighborhood $V$
			as in the previous part is sent to the singleton $\{[\alpha]\}$.
			Hence, all singletons are open in $\Rep(H,G)$.
		\item As each $[\alpha] \in \Rep(H,G)$ is open, so is its preimage
			$\alpha G$. The $\alpha G$ are compact, too, as they are the images
			of the compact sets $\{\alpha\} \times G$ under the action map
			$\map(H,G) \times G \to \map(H,G)$. Since $\map(H,G)$ is
			metrizable, it is Hausdorff, and the compact sets $\alpha G$ must
			be closed. Choosing $\{\alpha_i\}_{i \in I}$ as in the statement of
			this proposition, this implies that the canonical map
			\[
				\coprod_{i \in I} \alpha_i G \to \map(H,G)
			\]
			is a homeomorphism. The $G$-action on $\map(H,G)$ restricts to
			$G$-actions on the orbits $\alpha_i G$, and the canonical map is
			$G$-equivariant with respect to these actions.
	\end{enumerate}
\end{proof}

Of course, the above decomposition always exists on the set level. The new
insight is its compatibility with the topology on $\map(H,G)$. This allows us
to determine the weak homotopy type of $\orbgh(H,G)$:

\begin{prop}
	\label{prop:weakhtptypeoforbgh}
	Let $\{\alpha_i\}_{i \in I}$ be a choice of representatives of equivalence
	classes in $\Rep(H,G)$. Then the weak homotopy type of $\orbgh(H,G) =
	\map(H,G) \times_G EG$ is
	\[
		\coprod_{i \in I} B C (\alpha_i)
	\]
	where $C(\alpha_i) \subseteq G$ denotes the centralizer subgroup of the image of
	$\alpha_i\colon H \to G$.
\end{prop}
\begin{proof}
	We follow the argumentation in~\cite[Remark 2.2.1]{rezk_global_2014}: By
	Proposition~\ref{prop:maphgdecomp}.(iii), we have a $G$-equivariant
	decomposition $\map(H,G) \cong \coprod_{i \in I} \alpha_i G$. As the
	functor $- \times_G EG$ commutes with coproducts in the category of
	$G$-spaces, we obtain
	\[
		\orbgh(H,G) = \map(H,G) \times_G EG \cong 
		\coprod_{i \in I} \left( \alpha_i G \times_G EG \right),
	\]
	and it is left to verify that $\alpha_i G \times_G EG$ is a $B
	C(\alpha_i)$. Note that $\alpha_i G \cong C(\alpha_i) \backslash G$
	because $C(\alpha_i)$ is the stabilizer of $\alpha_i$ with
	respect to the $G$-action on $\map(H,G)$.

	The realization of the simplicial space $EG_\bullet$, given by
	$EG_n = G \times G^n$, is the space $EG$. It has a free $G$-action which
	is the realization of the $G$-action on $EG_\bullet$ through the first
	factor of the product. The restriction of this action to $C(\alpha_i)
	\subseteq G$ is free, so $C(\alpha_i) \backslash EG$ is a $B C(\alpha_i)$.

	We claim that $C(\alpha_i) \backslash EG$ is exactly $\alpha_i G \times_G
	EG = C(\alpha_i) \backslash G \times_G EG$. Restricting the $G$-action on
	$EG_\bullet$ to $C(\alpha_i)$ yields a $C(\alpha_i)$-action on
	$EG_\bullet$. The quotient by this action is the simplicial topological
	space with $n$-th level $C(\alpha_i) \backslash G \times G^n$ whose
	realization is precisely $C(\alpha_i) \backslash G \times_G EG$. Taking
	quotients by $C(\alpha_i)$-actions commutes with realization, so
	\[
		C(\alpha_i) \backslash G \times_G EG  = 
		|n \mapsto C(\alpha_i)\backslash G \times G^n| =
		|C(\alpha_i)\, \backslash\, EG_\bullet| \cong 
		C(\alpha_i) \backslash EG.
	\]
	Therefore, $C(\alpha_i) \backslash G \times_G EG$ is a $B C(\alpha_i)$,
	concluding the proof.
\end{proof}

\subsection{A Reminder on \texorpdfstring{$\orbs(H,G)$}{\orbsPDF(H,G)}}
\label{subsec:orbsrem}
We briefly recall some definitions from \cite{schwede_orbispaces_2017}:

\begin{definition}~
	\label{def:orbsgeneral}
	\begin{enumerate}
		\item For real inner product spaces $V,W$ of finite or countably
			infinite dimension, $\L(V,W)$ is the space of linear isometric
			embeddings topologized as in Observation~\ref{obs:Ltopology}.
			
			In the special case $V=W=\ri$, we abbreviate $\cL := \L(\ri,\ri)$
			and observe that composition endows $\cL$ with the structure
			of a unital topological monoid.
		\item A compact subgroup $G \subseteq \cL$ is a
			\emph{universal subgroup}
			\cite[Definition 1.4]{schwede_orbispaces_2017} if
			\begin{enumerate}
				\item it admits a necessarily unique structure of a Lie group,
				\item the induced orthogonal $G$-representation on $\ri$,
					denoted by $\ri_G$, is a complete $G$-universe, i.e., 
					every finite-dimensional $G$-representation embeds into
					$\ri_G$ via a $G$-equivariant linear isometry.
			\end{enumerate}
		\item If $H,G$ are such universal subgroups of $\cL$, we denote by
			$E(H,G) = \L(\ri_G,\ri_H)$ the right $H \times G$-space with
			underlying space $\cL = \L(\ri,\ri)$ and 
			\[
				(f \cdot (h,g))(x) := h^{-1} \cdot f (gx)
			\]
			for $f \in \L(\ri,\ri)$, $(h,g) \in H \times G$, and $x \in \ri$,
			using the $G$- and $H$-actions on $\ri$.
			
			Finally \cite[Definition~2.1]{schwede_orbispaces_2017},
			\[
				\orbs(H,G) := (\L(\ri_G,\ri_H) /G)^H = (E(H,G)/G)^H.
			\]
	\end{enumerate}
\end{definition}

\begin{rem}
	\phantomsection 
	\label{rem:compunivareall}
	\begin{enumerate}
		\item The concept of universal subgroups of $\cL$ from
			Definition~\ref{def:orbsgeneral} encompasses all compact Lie
			groups in the sense that isomorphism classes of compact Lie groups
			are in bijection with conjugacy classes of universal
			subgroups of $\cL$~\cite[Proposition~1.5]{schwede_orbispaces_2017}.
		\item An alternative definition of $\orbs(H,G)$ can be given by the space
			of $\cL$-equivariant maps
			\[
				\map_\cL (\cL/H,\cL/G),
			\]
			see also~\cite[Definition~2.1]{schwede_orbispaces_2017}.
			The evaluation at $[\id_{\ri}] \in \cL/H$ induces a
			homeomorphism from this space to our definition of $\orbs(H,G)$.
	\end{enumerate}
\end{rem}

The crucial property of $E(H,G)$, allowing for an identification of the
weak homotopy type of $\orbs(H,G)$, is its universality with respect to
\emph{graph subgroups}. 

\begin{definition}
	Let $H,G$ be compact Lie groups. 
	\begin{enumerate}
		\item The \emph{family of graph subgroups}
			(\cite[Definition 1.1.25]{schwede_global_2017})
			is defined to be
				\begin{align*}
					\cF (H,G) := \{&\graph(\alpha) \subseteq H \times G
						\sth \alpha\colon L \to G \text{ a Lie group} \\
					&\quad \text{homomorphism for some closed }
						L \leq H \}.
				\end{align*}
		\item If $E$ is an $H \times G$-space and $\alpha\colon L \to G$, $L \leq
			H$ closed, then we set $E^\alpha := E^{\graph(\alpha)}$ and, as
			in Proposition~\ref{prop:weakhtptypeoforbgh}, we let
			$C(\alpha) \subseteq G$ denote the centralizer of the image of
			$\alpha$.
	\end{enumerate}
\end{definition}
\begin{prop}[{\cite[Proposition~1.1.26.(i)]{schwede_global_2017} and
	\cite[Proposition~A.10]{schwede_orbispaces_2017}}]
	\label{prop:ehgunivspace}
	Given universal subgroups $H,G$ of $\cL$, the $H \times G$-space
	$E(H,G)$ is a universal space for the family $\cF(H,G)$, i.e., for $K \le
	H \times G$,
	\[
		E(H,G)^K \simeq
		\begin{cases}
			\ast & \text{if } K \in \cF(H,G), \\
			\emptyset & \text{if } K \notin \cF(H,G).
		\end{cases}
	\]
\end{prop}
\begin{prop}[{\cite[Proposition 1.5.12.(i)]{schwede_global_2017}}]
	\label{prop:orbsiscoproduct}
	For any choice of representatives
	$\{\alpha_i\}_{i \in I}$ of
	equivalence classes in $\Rep(H,G) = \map(H,G)/G$, the canonical
	map
	\[
		\lambda\colon \coprod_{i \in I} E(H,G)^{\alpha_i} / C(\alpha_i)
		\to (E(H,G)/G)^H
	\]
	is well-defined and a homeomorphism. Furthermore, $E(H,G)^{\alpha_i} /
	C(\alpha_i)$ is a $BC(\alpha_i)$.
\end{prop}

\begin{rem}~
	\label{rem:orbshgmisc}
	\begin{enumerate}
		\item Therefore, the weak homotopy type of $\orbs(H,G)$ is $\coprod_{i
			\in I} BC(\alpha_i)$ (see also~\cite[Remark
			2.2]{schwede_orbispaces_2017}). This agrees with the weak
			homotopy type of $\orbgh(H,G)$ by
			Proposition~\ref{prop:weakhtptypeoforbgh}.
		\item Of course, the map $\lambda$ from before depends on the groups
			$H,G$ and should be called $\lambda_{H,G}$. However, in order to
			reduce notational clutter, we chose to only decorate those maps
			with subscripts that will be considered again for varying
			groups $H,G$ as we discuss the composition law in the next
			section (such as the map $\widetilde{p}_{H,G}$ from
			Definition~\ref{def:ethg}). Henceforth, this principle will be
			applied without further remarks.
	\end{enumerate}
\end{rem}

\subsection{The Intermediate Mapping Space
	\texorpdfstring{$\orbh(H,G)$}{\orbhPDF(H,G)}}
	\label{subsec:intermediatemappingspace}

In order to compare $\orbgh(H,G)$ to $\orbs(H,G)$, we introduce a fattened up
version of $\map(H,G)$ that incorporates $E(H,G)$. Its homotopy quotient will
be $\orbh(H,G)$.

\begin{convention}
	For the remainder of this section, let $H$ and $G$ be universal
	subgroups of $\cL$ unless otherwise stated.
\end{convention}

\begin{definition}
	\label{def:ethg}
	The $G$-space $\et(H,G)$ has underlying set
	\[
		\{ (\alpha,x) \in \map(H,G) \times E(H,G)\sth x \in E(H,G)^\alpha
		\} 
	\]
	topologized as a subspace of $\map(H,G) \times E(H,G)$ and equipped with
	the diagonal action. It comes with a canonical projection map
	$\widetilde{p}_{H,G}\colon \et(H,G) \to \map(H,G)$.
\end{definition}
\begin{rem}
	\label{rem:ethgcgwh}
	Lemma~\ref{lem:ethgtopology} will show that this is a CGWH space.
\end{rem}
\begin{lemma}
	\label{lem:etactionwelldef}
	The $G$-action on $\et(H,G)$ is well-defined.
\end{lemma}
\begin{proof}
	Let $\alpha \in \map(H,G)$ and $x \in E(H,G)^\alpha$. We need to show that $xg
	\in E(H,G)^{\alpha g}$. Spelling out the definitions, this amount to verifying
	that for each $h \in H$, the point $xg = x \cdot (1,g)$ is fixed under
	the action of $(h,g^{-1}\alpha(h)g)$. We compute
	\[
		x \cdot (1,g) \cdot (h,g^{-1}\alpha(h)g) =
		x \cdot (h, \alpha(h) g) = 
		x \cdot (h,\alpha(h)) \cdot (1,g) =
		x \cdot (1,g) = xg.
	\]

	The $G$-action is automatically continuous because it is the restriction
	of the continuous diagonal action on $\map(H,G) \times E(H,G)$.
\end{proof}

\begin{prop}
	\label{prop:ptildefiberbundle}
	The map $\widetilde{p}_{H,G}$ is a fiber bundle.
\end{prop}
\begin{proof}
	Let us write $\widetilde{p} = \widetilde{p}_{H,G}$ within this proof.
	By Proposition~\ref{prop:maphgdecomp}, it suffices to show that the
	restriction
	\[
		 \widetilde{p}^{-1} (\alpha G) \to \alpha G
	\]
	is a fiber bundle for any $\alpha \in \map(H,G)$, where $\alpha G
	\subseteq \map(H,G)$ is the $G$-orbit of $\alpha$.

	The canonical map $\Stab (\alpha) \backslash G \to \alpha G$
	is a continuous bijection from a compact space to
	$\alpha G \subseteq \map(H,G)$. Since the space $\map(H,G)$ is metrizable,
	it is Hausdorff and so is its subspace $\alpha G$. Thus, the map is a
	homeomorphism. Therefore, the map $\widetilde{p}^{-1}(\alpha G) \to \alpha
	G$ from before is homeomorphic to the projection map
	\[
		\widetilde{p}_\alpha\colon
		\{ (\overline{g},x) \in (\Stab(\alpha) \backslash G) \times E(H,G)
			\sth x \in E(H,G)^{\alpha g} \} \to \Stab(\alpha) \backslash G.
	\]

	Fix a point $\overline{g} \in \Stab(\alpha) \backslash G$.
	The stabilizer $\Stab(\alpha) \subseteq G$ is a closed subgroup of the Lie
	group $G$. In particular, the quotient map $\pi\colon G \to \Stab(\alpha)
	\backslash G$ has local sections, and we can choose a neighborhood $U$ of
	$\overline{g}$ and a section $s\colon U \to G$. We are now in the
	position to write down a local trivialization of $\widetilde{p}_\alpha$:
	\[
		\def\arraystretch{1.5}
		\begin{array}{ccc}
			\widetilde{p}_\alpha^{-1} (U) = 
				\{ (\overline{h},x) \in U \times E(H,G) \sth
					x \in E(H,G)^{\alpha h} \} &
			\cong &
			U \times E(H,G)^\alpha,  \\
			(\overline{h},x) & \mapsto & (\overline{h},x \cdot
				s(\overline{h})^{-1}), \\
			(\overline{h},x \cdot s(\overline{h})) & \mapsfrom &
				(\overline{h},x).
		\end{array}
	\]
	The maps are continuous because the section $s$ is continuous and it is
	easy to see that they are mutually inverse. It is left to verify that they
	are well-defined. For $(\overline{h},x) \in U \times E(H,G)^\alpha$, we
	have to check that $x \cdot s(\overline{h}) \in E(H,G)^{\alpha h}$. As $s$
	is a section, we can write $s(\overline{h}) = t h$ for some $t \in
	\Stab(\alpha)$. Hence, $x \cdot s(\overline{h}) \in E(H,G)^{\alpha th} =
	E(H,G)^{\alpha h}$. Well-definedness for the other map works in the same
	fashion. Finally, it is evident that the homeomorphisms are compatible
	with the projections to $U$. Therefore, they constitute a local
	trivialization of $\widetilde{p}_\alpha$, concluding the proof.
\end{proof}

\begin{cor}
	The map $\widetilde{p}_{H,G}$ is a weak equivalence.
\end{cor}
\begin{proof}
	By the previous proposition, $\widetilde{p}_{H,G}$ is a Serre fibration.
	Over any point $\alpha \in \map(H,G)$, its fiber is exactly
	$E(H,G)^\alpha$, which is contractible by
	Proposition~\ref{prop:ehgunivspace}.  Consequently, $\widetilde{p}_{H,G}$
	is a weak equivalence.
\end{proof}

\begin{definition}
	\label{def:lhg}
	Using the bar construction as a model for $EG$, we define
	\[
		\orbh(H,G) := \et(H,G) \times_G EG.
	\]
	The map $\widetilde{p}_{H,G}\colon \et(H,G) \to \map(H,G)$ from
	Definition~\ref{def:ethg} induces a map
	\[
		l_{H,G}\colon \orbh(H,G) = \et(H,G) \times_G EG \to \map(H,G) \times_G EG =
		\orbgh(H,G)
	\]
	by passing to homotopy quotients
\end{definition}

\begin{prop}
	\label{prop:lhgwe}
	The map $l_{H,G}$ is a weak equivalence.
\end{prop}
\begin{proof}
	The weak equivalence $\widetilde{p}_{H,G}\colon \et(H,G) \to \map(H,G)$
	induces a levelwise weak equivalence
	\[
		\et(H,G) \times G^n \to \map(H,G) \times G^n
	\]
	between those simplicial topological spaces that realize to the homotopy
	quotients $\et(H,G)
	\times_G EG$ and $\map(H,G) \times_G EG$ respectively.

	The unit map $\ast \to G$ is a Hurewicz cofibration because $G$ is a
	smooth manifold. Therefore, any degeneracy map of the two simplicial
	topological spaces is a Hurewicz cofibration, and both simplicial
	topological spaces are good. Goodness implies properness, and levelwise
	weak equivalences between proper simplicial topological spaces realize to
	weak equivalences. Hence, the induced map $l_{H,G}$ on realizations is a
	weak equivalence.
\end{proof}

\subsection{The Map to \texorpdfstring{$\orbs(H,G)$}{\orbsPDF(H,G)}}
\label{subsec:orbhtoorbs}

To complete our desired zig-zag of mapping spaces, we still have to construct
a map from $\orbh(H,G)$ to $\orbs(H,G)$. At the end of this subsection, we will have
proven that $\et(H,G) / G = \orbs(H,G)$ so that we can use the canonical map
\[
	\orbh(H,G) = \et(H,G) \times_G EG \to \et(H,G) / G = \orbs(H,G).
\]

\begin{prop}
	\label{prop:muhomeo}
	Fix a set $\{\alpha_i\}_{i \in I}$ of representatives of equivalence
	classes in
	$\map(H,G) / G = \Rep(H,G)$. Then the canonical map
	\[
		\mu\colon \coprod_{i \in I} E(H,G)^{\alpha_i} / C(\alpha_i) \to
		\et(H,G)/G,
	\]
	given by its components $\mu_i([x]) = [\alpha_i,x]$, is a homeomorphism.
\end{prop}
\begin{proof}
	We have to show that the components $\mu_i$ are well-defined, implying
	continuity for $\mu$, and that $\mu$ is bijective with continuous inverse.

	For well-definedness, observe that $C(\alpha_i) = \Stab(\alpha_i)$ with
	respect to the action of $G$ on $\map(H,G)$. Pick $g \in \Stab(\alpha_i)$
	and $x \in E(H,G)^{\alpha_i}$. Then $\mu_i([xg]) = [\alpha_i,xg] =
	[\alpha_i g,xg] = \mu_i([x])$. Thus, $\mu_i$ is well-defined, and $\mu$ is
	automatically continuous.

	In order to prove surjectivity, pick $[\alpha,x]$ in the codomain of
	$\mu$. We have $\alpha = \alpha_i g$ for some $i \in I$ and $g \in G$. So,
	$[\alpha,x] = [\alpha_i g,x] = [\alpha_i,x g^{-1}] = \mu_i([x g^{-1}])$ is
	in the image of $\mu$.

	Considering injectivity, suppose that for $i,j \in I$, we have $x \in
	E(H,G)^{\alpha_i}$ and $y \in E(H,G)^{\alpha_j}$ such that $\mu_i([x]) =
	\mu_j ([y])$. This means that $[\alpha_i,x] = [\alpha_j,y]$ so there is $g
	\in G$ such that $(\alpha_i, x) = (\alpha_j g, y g) \in \et(H,G)$. In
	particular, $\alpha_i$ and $\alpha_j$ are in the same orbit and we must
	have $i=j$. Furthermore, $\alpha_i = \alpha_i g$ implies that $g \in
	\Stab(\alpha_i)$. As we have $x=y g$, we obtain $[x] = [y] \in
	E(H,G)^{\alpha_i} / \Stab(\alpha_i) = E(H,G)^{\alpha_i} / C(\alpha_i)$.

	It remains to show that the inverse map is continuous. To this end,
	consider the $G$-equivariant decomposition $\map(H,G) \cong \coprod_i
	\alpha_i G$, see Proposition~\ref{prop:maphgdecomp}. Each component
	$\alpha_i G$ is canonically homeomorphic to $\Stab(\alpha_i) \backslash
	G$. Set
	\[
		\et(H,G)_i := 
		\{(\overline{g},x) \in \Stab(\alpha_i) \backslash G \times E(H,G) \sth
			x \in E^{\alpha_i g} \}.
	\]
	This is a $G$-space via the diagonal action, and we have $\et(H,G) \cong
	\coprod_i \et(H,G)_i$ equivariantly. As coproducts commute with quotients, 
	the codomain of $\mu$
	is a coproduct of the $\et(H,G)_i / G$ with respect to the maps
	\[
		\begin{array}{rccc}
			\iota_i\colon & \et(H,G)_i / G & \to & \et(H,G)/G, \\
			& [\overline{g},x] & \mapsto & [\alpha_i g,x].
		\end{array}
	\]
	A computation reveals that the set-theoretic inverse of $\mu$ has
	components
	\[
		\begin{array}{rccc}
			\xi_i\colon & \et(H,G)_i /G & \to &
				E(H,G)^{\alpha_i} / \Stab(\alpha_i), \\
				& [\overline{g},x] & \mapsto & [x g^{-1}].
		\end{array}
	\]
	with respect to the coproduct decomposition of $\et(H,G)/G$
	from before. To conclude the proof, we need to show that the $\xi_i$ are
	continuous. This will be achieved by identifying their domains
	$\et(H,G)_i/G$ as certain quotients.

	Fix $i \in I$. Since the fixed points $E(H,G)^{\alpha_i}$ are closed in
	$E(H,G)$, the subset $G \times E(H,G)^{\alpha_i} \subseteq G \times
	E(H,G)$ is closed as well. The map $\rho\colon G \times E(H,G) \to G \times
	E(H,G)$, $(g,x) \mapsto (g,x g^{-1})$, is continuous. Hence, $A :=
	\rho^{-1} (G \times E(H,G)^{\alpha_i})$ must be closed. The set $A$ can
	easily be identified as
	\[
		A = \{(g,x) \in G \times E(H,G) \sth x \in E(H,G)^{\alpha_i g} \}.
	\]

	Next, let us consider the quotient map $p$ from $G \times E(H,G)$ to
	$(\Stab(\alpha_i) \backslash G) \times E(H,G)$.  The subset $A$ of $G
	\times E(H,G)$ is $p$-saturated, i.e., $p^{-1}(p(A)) = A$. To see this, we
	must show that $(g,x) \in A$ implies $(gh,x h) \in A$ for any $h \in
	G$. But from $x \in E(H,G)^{\alpha_i g}$, it follows that $x h \in
	E(H,G)^{\alpha_i g h}$ and $(gh,x h) \in A$. Now, since $A$ is closed
	and saturated, the quotient map $p$ restricts to a quotient map
	$p_{|A}\colon A \to p(A)$. Moreover, $p(A) = \et(H,G)_i$. In the
	commutative diagram
	\[
		\begin{tikzcd}
			(g,x) \arrow[r,start anchor=east,mapsto]
				\arrow[d,phantom,"\in" sloped] &
			{[x \cdot g^{-1}]} \arrow[d,phantom,"\in" sloped] \\
			A \dar \rar & E(H,G)^{\alpha_i} / \Stab(\alpha_i) \\
			p(A) \dar \urar &  \\
			p(A) / G \arrow[uur,swap,"\xi_i"]
		\end{tikzcd}
	\]
	the vertical maps are quotient maps, and the horizontal map is continuous.
	Therefore, $\xi_i$ is continuous, and $\mu$ is a homeomorphism because its
	set-theoretic inverse is the coproduct of the continuous maps $\xi_i$.
\end{proof}

\begin{prop}
	\label{prop:orbstriangle}
	Define $\nu_{H,G}\colon \et(H,G)/G \to (E(H,G)/G)^H$ by sending
	$[\alpha,x]$ to $[x]$.
	Then $\nu_{H,G}$ is a well-defined
	homeomorphism and fits into a commutative diagram
	\[
		\begin{tikzcd}
			& \et(H,G) / G \arrow[dd,"\nu_{H,G}"] \\
			\coprod_{i \in I} E(H,G)^{\alpha_i} / C(\alpha_i) 
				\arrow[ur,"\mu","\cong"']
				\arrow[dr,"\lambda"',"\cong"] \\
			& (E(H,G)/G)^H = \orbs(H,G)
		\end{tikzcd}
	\]
	where $\lambda$ is as in Proposition~\ref{prop:orbsiscoproduct}.
\end{prop}
\begin{proof}
	The map $\mu$ is a homeomorphism by Proposition~\ref{prop:muhomeo}, and
	$\lambda$ is a homeomorphism by Proposition~\ref{prop:orbsiscoproduct}.
	Hence, there is a unique homeomorphism $\nu_{H,G}$ making the diagram
	commute, and one readily checks that the formula is correct. 
\end{proof}

We have now completed the necessary preparations to define the map that was
promised in the beginning of this subsection.

\begin{definition}
	\label{def:khg}
	The map $k_{H,G}$ is the composition
	\[
		\begin{tikzcd}[cramped]
			\orbh(H,G) = \et(H,G) \times_G EG \rar &
			\et(H,G) \times_G \ast \cong \et(H,G) / G
				\arrow[r,"\cong","\nu_{H,G}"'] &
			\orbs(H,G).
		\end{tikzcd}
	\]
\end{definition}

\begin{prop}
	\label{prop:khgwe}
	The map $k_{H,G}$ is a weak equivalence.
\end{prop}
\begin{proof}
	We claim that the $G$-action on $\et(H,G)$ is free. Indeed, if
	there were $\alpha,g,x$ such that $(\alpha,x) g = (\alpha,x)$,
	then $xg = x$ would imply that $x \in E(H,G)^{\langle g\rangle}$
	where $\langle g\rangle \subseteq G$ denotes the cyclic group
	generated by $g$.  If $g \neq 1$, then $1 \times \langle g\rangle$
	cannot be a graph subgroup of $H \times G$. As $E(H,G)$ is a
	universal space for the family of graph subgroups, $x \in
	E(H,G)^{\langle g\rangle} = \emptyset$, a contradiction.  Hence,
	$g=1$ and the action is free as desired.
	
	The freeness of the action, Theorem~\ref{thm:freeactionT2space}, and
	Lemma~\ref{lem:ethgtopology} imply that the comparison map between
	the homotopy quotient and the point-set quotient is a weak equivalence.
	Thus, $k_{H,G}$ is a weak equivalence.
\end{proof}

For the reader's convenience, we restate the main result of this section:
There is a zig-zag of weak equivalences
\[
	\begin{tikzcd}
		\orbgh(H,G) &
		\orbh(H,G) \arrow[l,"\simeq","l_{H,G}"'] \arrow[r,"\simeq"',"k_{H,G}"] &
		\orbs(H,G).
	\end{tikzcd}
\]
The map $l_{H,G}$ and its properties have been established in
Subsection~\ref{subsec:intermediatemappingspace} while $k_{H,G}$ has been
discussed in this subsection.

\section{The Zig-Zag of Dwyer-Kan Equivalences}
\label{sec:zigzagdk}

In this section, we will upgrade our results on mapping spaces to Dwyer-Kan
equi\-valences. In particular, we will interpret the spaces $\orbgh(H,G)$,
$\orbh(H,G)$, and $\orbs(H,G)$ as mapping spaces of topologically enriched
categories and verify that the comparison maps from the previous section give
rise to functors sitting in a zig-zag

\[
\begin{tikzcd}
	\orbgh & \orbh \arrow[l,"\simeq"',"l"]
	\arrow[r,"\simeq","k"'] & \orbs
\end{tikzcd}
\]
of Dwyer-Kan equivalences.

\begin{definition}
	\label{def:oborb}
	The topological categories $\orbgh$, $\orbh$, and $\orbs$ all
	have the set of objects
	\[
		\{ G \subseteq \L(\ri,\ri) \sth G \text{ a universal subgroup}\},
	\]
	and mapping spaces as defined (or suggested by notation) in the previous
	section.
\end{definition}

It remains to specify composition laws and to verify that identity morphisms
exist. We will start by reinterpreting the mapping spaces of $\orbh$ and of
$\orbgh$ as realizations of topological groupoids. This allows for a sleek
definition of the respective compositions, and we will also be able to verify
compatibility with the maps $l_{H,G}$ easily.

Afterward, we will recall the composition law on $\orbs$ and check that it is
compatible with the maps $k_{H,G}$.  This will imply that our comparison
zig-zag from the previous section gives rise to a zig-zag of topologically
enriched functors. Recall the following

\begin{definition}
	\label{def:dk}
	Let $\catc$ be topologically enriched category. Then the ordinary
	category $\pi_0 \catc$ has $\ob \pi_0 \catc = \ob \catc$ and
	$(\pi_0 \catc)(c,c^\prime) = \pi_0 (\catc(c,c^\prime))$ with composition
	defined in the obvious way.

	A functor $f\colon \catc \to \catd$ of topologically enriched categories is a
	\emph{Dwyer-Kan equivalence} if the induced functor
	$\pi_0 f\colon \pi_0 \catc \to \pi_0 \catd$ is an equivalence of categories and
	$f$ is a weak equivalence on all mapping spaces.
\end{definition}

As we have already verified that the maps in our desired zig-zag of functors
are weak equivalences on mapping spaces, and since the object functions are
identities, we will immediately be able to deduce that the zig-zag of functors
is a zig-zag of Dwyer-Kan equivalences.

\begin{convention}
	For the remainder of this section, $L,K,H,G$ shall be universal
	subgroups of $\L(\ri,\ri)$ unless otherwise specified.
\end{convention}

\subsection{Compositions on \texorpdfstring{$\orbh$}{\orbhPDF} and on
	\texorpdfstring{$\orbgh$}{\orbghPDF}}
\label{subsec:comporbhorbgh}

Recall that a \emph{topological groupoid} $\cG$ consists of topological spaces
$\cG_0, \cG_1$ together with source and target maps $s,t\colon \cG_1 \to \cG_0$ and
a composition $\circ\colon \cG_1 \tensor[_s]{\times}{_t} \cG_1 \to \cG_0$, subject
to associativity and identity conditions. The category of topological
groupoids is denoted by $\grpdtop$.

\begin{definition}
	Let $G$ be a topological group and $X$ be a right $G$-space. Then the
	\emph{action groupoid $X /\!/ G$} is given by $(X /\!/ G)_0 = X$, $(X /\!/ G)_1
	= X \times G$ with $s(x,g) = x$, $t(x,g) = x \cdot g$ and
	$(x^\prime,g^\prime) \circ (x,g) = (x,g g^\prime)$.

	Given a $G$-equivariant map $f\colon K \to L$ of right $G$-spaces, we
	define a functor
	$f /\!/ G\colon\allowbreak K /\!/ G \to L /\!/ G$ of topological
	groupoids by $(f /\!/ G)_0 = f$ and $(f /\!/ G)_1 = f \times G$.
\end{definition}

Some straightforward computations prove
\begin{prop}
	\label{prop:hquotfromactgroupoid}
	The construction from the previous definition gives rise to a functor $-
	/\!/ G\colon G\text{-}\Top \to \grpdtop$. Furthermore, the following
	diagram commutes
	\[
		\begin{tikzcd}
			G\text{-}\Top \arrow[r,"- /\!/ G"] \arrow[rrr,bend right,"- \times_G EG"] &
			\grpdtop \arrow[r,"N_\bullet"] &
			\Top^\Delta \arrow[r,"|\cdot|"] &
			\Top
		\end{tikzcd}
	\]
	where $N_\bullet$ is the topologically enriched version of the nerve
	functor.
\end{prop}

Thus, we see that the morphism spaces of $\orbh$ and of $\orbgh$ can be
described as realizations of certain action groupoids, namely
\[
	\orbgh(H,G) = |N_\bullet (\map(H,G) /\!/ G)|, \quad
	\orbh(H,G) = |N_\bullet (\et(H,G) /\!/ G) |.
\]

\begin{definition}~
	\label{def:compingroupoids}
	We define two functors of topological groupoids which will yield
	composition laws:
	\begin{enumerate}
		\item The functor
			$\diamond\colon (\map(H,G) /\!/ G) \times (\map(K,H) /\!/ H) \to 
				\map(K,G) /\!/ G$
			is given by composition in level $0$ and by the formula
			\[
				\begin{array}{ccc}
					\map(H,G) \times G \times \map(K,H) \times H & \to 
						& \map(K,G) \times G, \\
					(\beta,g),(\alpha,h) & \mapsto & (\beta \circ \alpha, \beta(h) g)
				\end{array}
			\]
			in level $1$.
		\item The functor
			$\diamond\colon (\et(H,G) /\!/ G) \times (\et(K,H) /\!/ H) \to
				\et(K,G) /\!/ G$
			is defined by
			\[
				(\beta,y),(\alpha,x) \mapsto (\beta \circ \alpha,x \circ y)
			\]
			in level $0$ and by
			\[
				(\beta,y,g),(\alpha,x,h) \mapsto
					(\beta \circ \alpha, x \circ y,\beta(h)g)
			\]
			in level $1$ where $\beta \in \map(H,G)$, $y \in E(H,G)^\beta$,
			$g \in G$, $\alpha \in \map(K,H)$, $x \in E(K,H)^\alpha$, and $h
			\in H$. In both formulae, we use the composition of $\cL$ which is
			the underlying space of both $E(H,G)$ and $E(K,H)$.
	\end{enumerate}
\end{definition}

\begin{rem}
	Denote by $\B G = \ast /\!/ G$ the action groupoid associated to the unique
	action of $G$ on the one-point space $\ast$. Then one can show that
	$\map(H,G) /\!/ G$ is isomorphic to the internal mapping topological groupoid
	$\maptopgrpd(\B H,\B G)$, and the composition $\diamond$ from part (i) of
	the previous definition is precisely the composition coming from this
	enrichment of $\grpdtop$ over itself.
\end{rem}

\begin{prop}
	The maps from the previous definition are functors.
\end{prop}

\begin{proof}
	We have to verify well-definedness and compatibility with composition
	and identities for both candidate functors $\diamond$. We will do this for
	the second one only because the necessary verifications for the first one
	are part of the verifications for the second one.

	In order to check well-definedness in level $0$, we have to
	prove that for $\beta,y,\alpha,x$ as above, $x \circ y \in E^{\beta \circ
	\alpha}$. This means that $x \circ y$ is fixed under the action of
	$\graph(\beta \circ \alpha) \subseteq K \times G$. Pick $k \in K$. We
	would like to show that $(x \circ y)(k,\beta(\alpha(k))) = (x \circ y)$.
	Choose $v \in \ri$. Spelling out the left hand side and evaluating at $v$,
	we obtain
	\[
		(x \circ y)(k,\beta(\alpha(k)) (v) = 
		k^{-1} \cdot (x \circ y)(\beta(\alpha(k)) v) =
		k^{-1} \cdot (x ( y(\beta(\alpha(k))v))).
	\]
	We have $y \in E^\beta$ by assumption and $\alpha(k) \in H$. Thus,
	$(\alpha(k),\beta(\alpha(k))) \in \graph(\beta)$ and $y \cdot
	(\alpha(k),\beta(\alpha(k))) = y$. Concretely,
	\[
		\alpha(k)^{-1} \cdot y(\beta(\alpha(k))v) = y(v),
	\]
	and, therefore, $y(\beta(\alpha(k))v) = \alpha(k) y(v)$. Plugging this
	into the previous computation,
	\[
		(x \circ y)(k,\beta(\alpha(k)))(v) = k^{-1} (x (\alpha(k)y(v))).
	\]
	As $x \in E^\alpha$, we have $x \cdot (k,\alpha(k)) = x$. Evaluating at
	$y(v) \in \ri$ proves that the right hand side becomes $x(y(v))$, and we
	have proven that $x \circ y \in E^{\beta \circ \alpha}$.

	For the well-definedness in level $1$, it is left to verify that the
	morphism of the topological groupoid $\et(K,G) /\!/ G$ specified by $(\beta
	\circ \alpha, x \circ y, \beta(h)g)$ has the correct source and the
	correct target. The source is
	$(\beta \circ \alpha, x \circ y) = s(\beta,y,g) \diamond s(\alpha,x,h)$ as
	desired. The target is $((\beta \circ \alpha) \cdot (\beta(h)g),(x \circ
	y)(\beta(h)g))$ and this should agree with $t(\beta,y,g) \diamond
	t(\alpha,x,h) = (\beta \cdot g, y \cdot g) \diamond (\alpha \cdot h,x \cdot
	h) = ((\beta \cdot g) \circ (\alpha \cdot h),(x\cdot h) \circ (y \cdot
	g))$. One readily calculates that $(\beta \circ \alpha) \cdot (\beta(h)g)
	= (\beta \cdot g) \circ (\alpha \cdot h) \in \map(K,G)$, so the first
	components agree. For the second components, pick $v \in \ri$. We have
	\[
		((x \cdot h) \circ (y \cdot g))(v) = (x \cdot h)(y(gv)) = x (h y(gv)).
	\]
	As before, $y \in E^\beta$, and $h y(gv) = y(\beta(h) gv)$ consequently.
	Plugging this in,
	\[
		((x \cdot h) \circ (y \cdot g))(v) = x(y(\beta(h) gv)) =
		(x \circ y)(\beta(h)gv) = ((x \circ y)(\beta(h)g))(v).
	\]
	Thus, the candidate functor $\diamond$ is compatible with sources and
	targets.

	Given an object $((\beta,y),(\alpha,x))$ in the product groupoid on the
	left hand side, its identity is the morphism
	$((\beta,y,1),(\alpha,x,1))$, and applying $\diamond$ yields $(\beta \circ
	\alpha, x \circ y,1)$ which is the identity of the object $(\beta \circ
	\alpha,x \circ y)$.

	Pick two composable morphisms in $(\et(H,G) /\!/ G) \times (\et(K,H) /\!/ H)$.
	Necessarily, they are of the form $((\beta,y,g),(\alpha,x,h))$ and $((\beta
	\cdot g, y \cdot g, g^\prime),(\alpha \cdot h, x \cdot h,h^\prime))$.
	Their composition inside the source groupoid is
	$((\beta,y,g g^\prime),(\alpha,x,h h^\prime))$ which is mapped to $(\beta
	\circ \alpha,x \circ y, \beta(hh^\prime)gg^\prime)$ by $\diamond$. On the
	other hand, the composition of
	the individual images under $\diamond$ inside the target groupoid is
	\[
		\begin{aligned}
			&((\beta \cdot g) \circ (\alpha \cdot h),(x \cdot h) \circ
				(y \cdot g), ((\beta \cdot g)(h^\prime)) g^\prime)
			\circ
			(\beta \circ \alpha,x \circ y,\beta(h)g) \\
			=& (\beta \circ \alpha,x \circ y,
				\beta(h)g ((\beta \cdot g)(h^\prime)) g^\prime) \\
			=& (\beta \circ \alpha,x \circ y,
			\beta(h)g g^{-1} \beta(h^\prime) g g^\prime),
		\end{aligned}
	\]
	which agrees with $(\beta \circ \alpha,x\circ
	y,\beta(hh^\prime)gg^\prime)$.
\end{proof}
\begin{prop}
	\label{prop:catenrichedingroupoids}
	The composition laws $\diamond$ defined before give rise to the categories
	$\orbgh_\grpd$ and $\orbh_\grpd$ enriched in topological groupoids with
	\[
		\ob \orbgh_\grpd = \ob \orbh_\grpd = \ob \orbgh = \ob \orbh,
	\]
	and
	\[
		\orbgh_\grpd (H,G) = \map(H,G) /\!/ G, \quad
		\orbh_\grpd(H,G) = \et(H,G) /\!/ G.
	\]
	Furthermore, the maps $\widetilde{p}_{H,G}\colon \et(H,G) \to \map(H,G)$
	from Definition~\ref{def:ethg} give rise to a functor
	$p\colon \orbh_\grpd \to \orbgh_\grpd$ which is the identity on objects
	and $p_{H,G} := \widetilde{p}_{H,G} /\!/ G$ on morphism groupoids.
\end{prop}
\begin{proof}
	We have to verify that the composition laws $\diamond$ are associative and
	unital. Again, we will do this for the second case $\orbh_\grpd$ only.

	Unitality means that there is a, necessarily unique, functor $\ast \to
	\orbh_\grpd (G,G)$ for the topological groupoid $\ast$ consisting of one
	object $\ast$ and its identity. Specifying such a functor amounts to the
	selection of an object 
	\[
		\id_G \in \orbh_\grpd (G,G)_0 = \et(G,G) = \{ (\alpha,x) \in \map(G,G)
			\times E(G,G) \sth x \in E(G,G)^\alpha \}.
	\]
	The underlying space of $E(G,G)$ is just $\cL = \L(\ri,\ri)$, and it is
	easily verified that $\id\colon \ri \to \ri \in \cL$ lives in the subspace
	$E^{\id}$. Thus, we claim that $\id_G := (\id,\id)$ yields a unit for
	$\diamond$.

	Let $(\alpha,x) \in \orbh_\grpd (H,G)$. Then $\id_G \diamond (\alpha,x) =
	(\id,\id) \diamond (\alpha,x) = (\id \circ \alpha,x \circ \id) =
	(\alpha,x)$ as desired. Similarly, $(\alpha,x) \diamond \id_H =
	(\alpha,x)$. On the level of morphisms, the functor $\ast \to
	\orbh_\grpd (G,G)$ selects the identity of $\id_G$ which is the morphism 
	$(\id,\id,1) \in \orbh_\grpd (G,G) (\id_G,\id_G) \subseteq \et(G,G) \times
	G$. For a morphism $(\alpha,x,h) \in \orbh_\grpd(H,G)$, we compute
	$(\id,\id,1) \diamond (\alpha,x,h) = (\id \circ \alpha,x \circ
	\id,\id(h)1) = (\alpha,x,h)$ as desired. To verify unitality from the
	right, we have
	$(\alpha,x,h) \diamond (\id,\id,1) = (\alpha,x,\alpha(1)h) = (\alpha,x,h)$
	using the identity of $\id_H$.

	We prove associativity on the level of morphisms only because the
	statement on objects follows from that. Thus, pick elements $(\alpha,x,k) \in
	\orbh_\grpd(L,K)_1$, $(\beta,y,h) \in \orbh_\grpd(K,H)_1$, and
	$(\gamma,z,g) \in \orbh_\grpd(H,G)_1$. We calculate
	\[
		\begin{aligned}
			&\left( (\gamma,z,g) \diamond (\beta,y,h) \right)
				\diamond (\alpha,x,k) \\
			=& (\gamma \circ \beta,y \circ z, \gamma(h)g) \diamond
				(\alpha,x,k) \\
			=& (\gamma \circ \beta \circ \alpha, x \circ y \circ z,
				(\gamma \circ \beta)(k)\gamma(h)g) \\
			=& (\gamma \circ \beta \circ \alpha, x \circ y \circ z,
				\gamma (\beta(k)h)g) \\
			=& (\gamma,z,g) \diamond
				(\beta \circ \alpha,x \circ y,\beta(k)h) \\
			=& (\gamma,z,g) \diamond
				\left( (\beta,y,h)\diamond (\alpha,x,k)\right)
		\end{aligned}
	\]
	concluding the verification of associativity.

	Spelling out the definitions, we see that the map $p_{H,G} =
	\widetilde{p}_{H,G} /\!/ G$ of topological groupoids is given by
	\[
		\begin{array}[]{ccc}
			\orbh_\grpd(H,G)_0 & \to & \orbgh_\grpd(H,G)_0, \\
			(\alpha,x) & \mapsto & \alpha,
		\end{array}
	\]
	\[
		\begin{array}{ccc}
			\orbh_\grpd(H,G)_1 & \to & \orbgh_\grpd(H,G)_1, \\
				(\alpha,x,g) & \mapsto & (\alpha,g) ,
		\end{array}
	\]
	and it is evident that these formulae are compatible with the composition
	laws in $\orbgh_\grpd$ and $\orbh_\grpd$, respectively, as $H,G$ vary.
	Therefore, the maps $p_{H,G}$ assemble into a functor
	$p\colon \orbh_\grpd \to \orbgh_\grpd$ as claimed.
\end{proof}

\begin{prop}
	\label{prop:lisdkeq}
	Applying the functor $|\cdot| \circ N_\bullet$ as in
	Proposition~\ref{prop:hquotfromactgroupoid} to the composition laws from
	Definition~\ref{def:compingroupoids} gives rise to categories $\orbh$ and
	$\orbgh$ enriched in topological spaces.

	Moreover, the functor $p$ from
	Proposition~\ref{prop:catenrichedingroupoids} yields a topologically
	enriched functor $l\colon \orbh \to \orbgh$ whose mapping space components
	$l_{H,G}\colon \orbh(H,G) \to \orbgh(H,G)$ are exactly the maps
	$l_{H,G}$ from Definition~\ref{def:lhg}. In particular, the functor $l$ is
	a Dwyer-Kan equivalence.
\end{prop}
\begin{proof}
	Both $|\cdot|$ and $N_\bullet$ are strongly monoidal. This is
	straightforward for the functor $N_\bullet$. A reference for the
	monoidality of $|\cdot|$ is \cite[Corollary 11.6]{may_geometry_1972}.
	
	Therefore, applying these functors to the individual mapping spaces of
	$\orbh_\grpd$ and $\orbgh_\grpd$ and keeping the object sets the same
	turns these categories enriched over topological groupoids into categories
	$\orbh$ and $\orbgh$ enriched over topological spaces.

	We have $|N_\bullet \orbh_\grpd(H,G)| = |N_\bullet (\et(H,G) /\!/ G)| =
	\et(H,G) \times_G EG$ and, similarly, $|N_\bullet \orbgh_\grpd(H,G)| =
	\map(H,G) \times_G EG$ by
	Proposition~\ref{prop:hquotfromactgroupoid}. Hence, the mapping spaces
	$\orbh(H,G)$ and $\orbgh(H,G)$ agree with our definitions from
	Section~\ref{sec:intermmappingspace}. 

	Applying $|\cdot| \circ N_\bullet$ to $p$ yields a functor
	$l\colon \orbh \to \orbgh$. As the mapping groupoid components
	$p_{H,G}\colon \orbh_\grpd (H,G) \to \orbgh_\grpd (H,G)$ of $p$ come from
	maps $\widetilde{p}_{H,G}$ of $G$-spaces, we may apply
	Proposition~\ref{prop:hquotfromactgroupoid} again and deduce that the
	mapping space components of $l = |N_\bullet (p)|$ are $
	|N_\bullet(p)|_{H,G} = | N_\bullet (p_{H,G}) | = \widetilde{p}_{H,G}
	\times_G EG$ which agree with the maps $l_{H,G}$ from
	Definition~\ref{def:lhg}. These are weak equivalences by
	Proposition~\ref{prop:lhgwe}. Consequently, $l$ is a Dwyer-Kan
	equivalence.
\end{proof}

\begin{rem}
	\label{rem:orbghfatunclear}
	In \cite[Remark 4.3]{gepner_homotopy_2007}, Gepner and Henriques choose
	$\orbgh(H,G)$ to be the \emph{fat} geometric realization $||N_\bullet
	(\map(H,G) /\!/ G)||$.  Let us denote this space by
	$\orbgh_{||\cdot||}(H,G)$ for the moment.
	For any good simplicial space $X$, the canonical map $||X|| \to |X|$ is a
	homotopy equivalence, see~\cite[Proposition
	A.1.(iv)]{segal_categories_1974}. In particular, $\orbgh_{||\cdot||}(H,G)
	\to \orbgh(H,G)$ is a homotopy equivalence.
	
	The composition in
	$\orbgh_{||\cdot||}$ depends on a choice of maps
	$i\colon ||X \times Y || \to ||X|| \times ||Y||$ and
	$r\colon ||X|| \times ||Y|| \to ||X \times Y ||$ such that
	$||X \times Y||$ is a retract of $||X|| \times ||Y||$ via these maps and
	such that certain associativity conditions hold, see \cite[Remarks 2.23
	and 4.3]{gepner_homotopy_2007}.
	If one chooses these maps in such a way that the diagram
	\[
		\begin{tikzcd}
			{||X|| \times ||Y||} \arrow[r,"r","\simeq"'] \dar{\simeq} &
			{||X \times Y||} \dar{\simeq} \\
			{|X| \times |Y|} \arrow[r,"\cong"] &
			{|X \times Y|}
		\end{tikzcd}
	\]
	commutes, then one obtains a Dwyer-Kan equivalence
	$\orbgh_{||\cdot||} \to \orbgh$.
\end{rem}
\subsection{Compatibility with the Composition on
	\texorpdfstring{$\orbs$}{\orbsPDF}}
\label{subsec:comporbhorbs}

We begin by recalling the composition on $\orbs$. Having established explicit
descriptions of the maps and objects in question, it will be straightforward
to verify that the remaining maps $k_{H,G}$ are compatible with identities and
composition.

\begin{definition}[{\cite[Definition~2.1]{schwede_orbispaces_2017}}]
	The topologically enriched category $\orbs$ has objects as in
	Definition~\ref{def:oborb} and mapping spaces $\orbs(H,G) = (E(H,G)/G)^H$
	as in Definition~\ref{def:orbsgeneral}. Composition is defined as
	\[
		\begin{array}[]{ccc}
			\orbs(H,G) \times \orbs(K,H) & \to & \orbs(K,G), \\
			([y] , [x]) & \mapsto & [x \circ y].
		\end{array}
	\]
	The identity morphism in $\orbs(G,G)$ is given by $[\id_{\ri}]$, the class
	of the neutral element of the topological monoid $\cL = \L(\ri,\ri)$,
	which is the underlying space of $\et(G,G)$.
\end{definition}

\begin{prop}
The maps $k_{H,G}\colon \orbh(H,G) \to \orbs(H,G)$ from
Definition~\ref{def:khg} give rise to a functor $k\colon \orbh \to \orbs$ that is
the identity on objects.
\end{prop}
\begin{proof}
	Recall that $k_{H,G}$ is defined as the composition	
	\[
		\begin{tikzcd}[cramped]
			\orbh(H,G) = \et(H,G) \times_G EG \rar &
			\et(H,G) \times_G \ast \cong \et(H,G) / G
				\arrow[r,"\cong","\nu_{H,G}"'] &
			\orbs(H,G)
		\end{tikzcd}
	\]
	where $\nu_{H,G}$ is given by $[\alpha,x] \mapsto [x]$, see
	Proposition~\ref{prop:orbstriangle}, and the first map is given by
	$[\alpha,x,u] \mapsto [\alpha,x]$.

	We have used the bar construction model $|N_\bullet (\et(H,G) /\!/ G)|$ for
	$\et(H,G) \times_G EG = \orbh(H,G)$, and an arbitrary element of this space
	is of the form $[\alpha,x,\vec{g},s]$ for $(\alpha,x) \in \et(H,G)$,
	$\vec{g} \in G^n$, $s \in \Delta^n$. The map $\et(H,G) \times_G EG \to
	\et(H,G) / G$ becomes $[\alpha,x,\vec{g},s] \mapsto [\alpha,x]$ under this
	implicit isomorphism. We conclude that the composite map $k_{H,G}$ has the
	formula $[\alpha,x,\vec{g},s] \mapsto [x]$.

	In this description, the composition on $\orbh$, which is given by
	realizing the maps $\diamond$ from Definition~\ref{def:compingroupoids},
	is given by
	\[
		\begin{array}[]{ccc}
			\orbh(H,G) \times \orbh(K,H) & \to & \orbh(K,H), \\
			{[\beta,y,\vec{g},t],[\alpha,x,\vec{h},s]} & \mapsto &
				[\beta \circ \alpha,x \circ y,?]
		\end{array}
	\]
	where the unspecified value $?$ is irrelevant to our following
	considerations.

	Namely, we need to verify that $k$ is compatible with identities and
	composition. We leave the straightforward computation for identities to the
	inclined reader and proceed to the composition. Consider the following
	diagram:
	\[
		\begin{tikzcd}
			\orbh(H,G) \times \orbh(K,H)
				\rar{\circ_{K,H,G}} \dar{k_{H,G} \times k_{K,H}} &
			\orbh(K,G) \dar{k_{K,G}} \\
			\orbs(H,G) \times \orbs(K,H) \rar{\circ_{K,H,G}} & \orbs(K,G)
		\end{tikzcd}
	\]
	Starting with $[\beta,y,\vec{g},t],[\alpha,x,\vec{h},s]$ in the upper
	left corner, we check that it is mapped to $[x \circ y]$ under $k_{K,G}
	\circ (\circ_{K,H,G})$ which is equal to its value under $\circ_{K,H,G}
	\circ (k_{H,G} \times k_{K,H})$. This concludes the proof.
\end{proof}

\begin{cor}
	\label{cor:zigzagdk}
	There is a zig-zag of Dwyer-Kan equivalences
	\[
		\begin{tikzcd}
			\orbgh & \orbh \arrow[l,"\simeq"',"l"]
			\arrow[r,"\simeq","k"'] & \orbs
		\end{tikzcd}
	\]
	as promised at the beginning of this section.
\end{cor}
\begin{proof}
	The functor $l$ exists and is a Dwyer-Kan equivalence by
	Proposition~\ref{prop:lisdkeq}. The existence of $k$ was discussed in the
	previous Proposition. It is a Dwyer-Kan equivalence by
	Proposition~\ref{prop:khgwe}.
\end{proof}

\section{Generalizations of the Main Results}
\label{sec:results}

\subsection{\texorpdfstring{$\cF$}{F}-Relative Versions}
\label{subsec:relative}

At the beginning of the introduction, we have described global homotopy
theory with respect to the class of all compact Lie groups. More generally,
one can study global homotopy theory with respect to a chosen \emph{global
family} $\cF$, i.e., a non-empty class of compact Lie groups which is closed
under isomorphisms, subgroups, and quotients,
see~\cite[Remark~3.11]{schwede_orbispaces_2017}. An $\cF$-global space should
be thought of as a space that is equipped with simultaneous and compatible
actions of all $G \in \cF$, and $\cF$-global equivalences are defined using
$G$-fixed points for $G \in \cF$.

Aside from the class $\cF_\text{all}$ of all compact Lie groups, examples for
global families are given by compact abelian Lie groups, finite groups, or
trivial groups, respectively.

In our setup, the set of objects of both $\orbs$ and $\orbgh$ is given by the
set of all universal subgroups $G \subseteq \L(\ri,\ri)$. While
this set is not isomorphism-closed, it contains a representative for every
isomorphism class of compact Lie groups, see Remark~\ref{rem:compunivareall}.
In particular, we can define versions of $\orbs$ and $\orbgh$ relative to a
global family and generalize our comparison result: Let $\cF$ be a global
family, and define $\obo_\cF$ to be the intersection of $\ob \orbs = \ob \orbgh$
with $\cF$. Note that for the global family of all compact Lie groups, we
have $\obo_{\cF_{\text{all}}} = \ob \orbs = \ob \orbgh$.

\begin{cor}
	Let $\obo$ be a subset of $\ob \orbs = \ob \orbgh$ and denote by
	$\orbs^\obo$ and $\orbgh^\obo$ the full subcategories of $\orbs$ and
	$\orbgh$, respectively, on this set of objects.

	There is a zig-zag of two Dwyer-Kan equivalences, both of which are the
	identity on objects, between $\orbs^\obo$ and $\orbgh^\obo$. This zig-zag
	induces a zig-zag of Quillen equivalences between the projective model
	structures on the enriched presheaf categories $\Pre(\orbs^\obo,\Top)$ and
	$\Pre(\orbgh^\obo,\Top)$.
	In particular, this is true for $\obo = \obo_\cF$ where $\cF$ is a global
	family.
\end{cor}

Given a global family $\cF$, the category $\orbs^{\obo_\cF}$ is called
$\orbs^\cF$ in~\cite[Remark~3.11]{schwede_orbispaces_2017}. There, Schwede
also equips both the category of $\cL$-spaces and the category of orthogonal
spaces with model structures relative to $\cF$ and establishes that these are
equivalent to the projective model structure on $\Pre(\orbs^\cF,\Top)$. This
yields three different models for $\cF$-global homotopy theory.

The setup of Gepner and Henriques is based on the choice of a \emph{family of
allowed isotropy groups} (\cite[Subsection~1.3]{gepner_homotopy_2007}) that
allows for even more general classes of groups which do not necessarily
consist of compact Lie groups only. Any global family $\cF$ is a family of
allowed isotropy groups in their sense, and the category
$\Pre(\orbgh^{\obo_\cF},\Top)$ is a model for orbispaces with
respect to $\cF$ in their setup,
see~\cite[Subsection~4.1]{gepner_homotopy_2007}.
In particular, the cited
paper compares this category with two different models for $\cF$-global
homotopy theory.

\subsection{Monomorphism Variants}

Another variant of orbispaces can be obtained by modifying the morphism
spaces of the orbit category as discussed
in~\cite[Subsection~2.1]{gepner_homotopy_2007}.
The morphism space $\orbgh(H,G)$ is constructed from the space
$\map(H,G)$ of \emph{all} morphisms from $H$ to $G$ by taking the homotopy
quotient by the $G$-action. Gepner and Henriques define another orbit
category by taking the subspace $\mono(H,G) \subseteq
\map(H,G)$ of monomorphisms and using
$\orbgh^\mono(H,G) := \mono(H,G) \times_G EG$ as the space of
maps from $H$ to $G$ in a new index category $\orbgh^\mono$. 

In order to define a monomorphism variant of $\orbs(H,G)$, 
consider a decomposition $\map(H,G) \cong \coprod_{i \in I} \alpha_i G$ as in
Proposition~\ref{prop:maphgdecomp}.
Since the subspace $\mono(H,G) \subseteq \map(H,G)$
is invariant under the $G$-action, there exists a subset $J \subseteq I$
such that $\mono(H,G) \cong \coprod_{i \in J} \alpha_i G$. By
Proposition~\ref{prop:orbsiscoproduct}, 
there is a homeomorphism
\[
	\lambda\colon \coprod_{i \in I} E(H,G)^{\alpha_i} / C(\alpha_i)
	\to (E(H,G)/G)^H = \orbs(H,G), 
\]
and we define $\orbs^\mono(H,G)$ to be the image of $\coprod_{i \in J}
E(H,G)^{\alpha_i} / C(\alpha_i)$ under $\lambda$.
This construction was explained to the author by Stefan Schwede.

It remains to find an analogon of $\orbh(H,G)$ for monomorphisms to complete
the zig-zag. Similarly to Subsection~\ref{subsec:intermediatemappingspace},
set
\[
	\et^\mono(H,G) :=
		\{(\alpha,x) \in \mono(H,G) \times E(H,G) \sth x \in E(H,G)^\alpha \},
\]
and $\orbhmono(H,G) := \et^\mono(H,G) \times_G EG$.

By imposing composition laws in analogy to Section~\ref{sec:zigzagdk}, we
obtain topologically enriched categories $\orbgh^\mono, \orbs^\mono$, and
$\orbhmono$ which have the same set of objects as the non-monomorphism
versions. The proof of Corollary~\ref{cor:zigzagdk} can be adapted
without complications to yield
\begin{cor}
	There is a zig-zag
	\[
		\begin{tikzcd}
			\orbgh^\mono & \orbhmono \arrow[l,"\simeq"']
			\arrow[r,"\simeq"] & \orbs^\mono
		\end{tikzcd}
	\]
	of Dwyer-Kan equivalences, 
	both of which are the identity on objects. This zig-zag
	induces a zig-zag of Quillen equivalences between the projective model
	structures on
	$\Pre(\orbgh^\mono,\Top)$ and $\Pre(\orbs^\mono,\Top)$.
\end{cor}

As in the previous subsection, we can restrict to a fixed subset $\obo$ of objects in
order to account for global families other than $\cF_\text{all}$:

\begin{cor}
	Let $\obo$ be a subset of $\ob \orbs^\mono = \ob \orbgh^\mono$ and denote by
	$\orbs^{\obo, \mono}$ and $\orbgh^{\obo, \mono}$ the full subcategories of
	$\orbs^\mono$ and $\orbgh^\mono$, respectively, on this set of objects.

	There is a zig-zag of two Dwyer-Kan equivalences, both of which are the
	identity on objects, between $\orbs^{\obo, \mono}$ and
	$\orbgh^{\obo, \mono}$. This zig-zag
	induces a zig-zag of Quillen equivalences between the projective model
	structures on the enriched presheaf categories $\Pre(\orbs^{\obo,
	\mono},\Top)$ and $\Pre(\orbgh^{\obo, \mono},\Top)$.
	In particular, this is true for $\obo = \obo_\cF$ where $\cF$ is a global
	family.
\end{cor}

\appendix

\section{Point-Set Topology}
This appendix is concerned with several technical topics which be deemed to be too
distracting for the flow of arguments to be included in the main part of this
exposition.

After stating our conventions on the usage of the adjectives \emph{compact}
and \emph{compactly generated} in Subsection~\ref{subsec:cgwhspaces}, we proceed by
providing helpful statements on the interaction of closed inclusions
with colimits in CGWH spaces in Subsection~\ref{subsec:closedincl}.
Afterward, we will recall \emph{normal spaces} in
Subsection~\ref{subsec:normal} and show that $EG$ is normal as a preparatory
result.

Subsection~\ref{subsec:freeactionhomquot} will show that for a free $G$-space,
which is Hausdorff, the comparison map from the homotopy quotient to the
point-set quotient is a weak equivalence
(Theorem~\ref{thm:freeactionT2space}).  This result is well-known for
manifolds or free $G$-CW-complexes, but we could not find a sufficiently
general statement in the literature that applies to our circumstances.
Therefore, we provide a proof that works for every compactly generated
Hausdorff space.

In the remaining Subsection~\ref{subsec:ethgtop}, we will investigate the
topology of the spaces $\et(H,G)$ in order to finalize some proofs from the
main body of this paper.

\subsection{Compactly Generated Weak Hausdorff Spaces}
\label{subsec:cgwhspaces}
The main body of this paper takes place in the category of compactly generated
weak Hausdorff spaces, also referred to as CGWH spaces. Before we deal with the
necessary point-set arguments, let us make the used terminology precise.

A space is \emph{compact} if every open cover admits a finite subcover. This
is also being referred to as \emph{quasi-compact} in other sources which
include the Hausdorff property into the definition of compactness.

Moreover, a space $X$ is \emph{compactly generated} if, for any subset $Y \subseteq
X$, $Y$ is closed if and only if $u^{-1} (Y)$ is closed for every compact
Hausdorff $K$ and every continuous $u\colon K \to X$. The space $X$ is weak
Hausdorff if for every such $u$ and $K$, the image $u(K)$ is closed
in $X$.

These definitions are taken from~\cite{strickland_category_2009}. Note that
this terminology varies within the literature, and some sources refer to
compactly generated spaces as \emph{$k$-spaces} while they take compactly generated
spaces to be compactly generated weak Hausdorff spaces in our sense.

Note that the property of being compactly generated is a local property, i.e.,
a space is compactly generated if and only if each point has a compactly
generated neighborhood. The property of being weak Hausdorff is not local,
though.

In this paper, we refer to CGWH spaces as \emph{(topological) spaces} and denote
the corresponding category by $\Top$. Within the next subsections, we will
have to deal with their point-set subtleties and cite statements about not
necessarily CGWH spaces. To this end, we will use the term \emph{general
topological space} for a space that is not necessarily CGWH.

The category of CGWH spaces is cocomplete. Limits and colimits may, however,
differ from those computed in the category of general topological spaces. Our
convention is that limits and colimits are computed in CGWH unless it is
explicitly declared that the diagram in consideration lives in the category
of general topological spaces. In this case, limits and colimits are to be
taken in the category of general topological spaces. The latter situation does
only occur within this Appendix.

For the special case of products, we adopt the following notation
from~\cite{strickland_category_2009}: Given two spaces $X$ and $Y$, we denote
by $X \times_0 Y$ the product taken in the category of general topological
spaces, which is not necessarily compactly generated. In contrast, $X \times
Y$ shall denote the product in the category of CGWH spaces.

\subsection{Closed Inclusions and CGWH Colimits}
\label{subsec:closedincl}

We will now shed some light on situations where specific colimits
agree regardless of whether they are computed in CGWH or in the category of
general topological spaces.

\begin{lemma}[{\cite[Section 2.4, p.~59]{hovey_model_1999}}]
	\label{lem:gentopspacessomecolimitsagree}
	The category of topological spaces is cocomplete. In the case of pushouts
	along closed inclusions or transfinite compositions of injections,
	colimits may be computed in the category of general topological spaces
	since they are already weak Hausdorff.
\end{lemma}

\begin{lemma}
	\label{lem:closedinclusionsclosed}
	In the category of topological spaces, closed inclusions are closed under
	pushouts, transfinite compositions, and retracts.
\end{lemma}
\begin{proof}
	As weak Hausdorff spaces are automatically $T_1$, a closed inclusion in the
	category of topological spaces is a closed $T_1$ inclusion in the category of
	general topological spaces. Retracts of maps of topological spaces are also
	retracts of maps of general topological spaces. Also, the relevant pushouts and
	transfinite compositions can be computed in the category of general
	topological spaces.

	The claim follows from the proof of \cite[Lemma 2.4.5]{hovey_model_1999}
	for the cases of pushouts and transfinite compositions and from the proof
	of \cite[Corollary 2.4.6]{hovey_model_1999} for the case of retracts.
\end{proof}

\subsection{Normal Spaces}
\label{subsec:normal}

In this subsection, we will recall the definition of a normal space.
Afterward, we will show that normal spaces are abundant and that the class of
normal spaces, in contrast to the class of Hausdorff spaces, is closed under
certain colimit constructions. As normality implies Hausdorffness for $T_1$
spaces, we will use our results to deduce that $EG$ and $\et(H,G)$ are
Hausdorff spaces (Lemma~\ref{lem:egnormal} and Lemma~\ref{lem:ethgtopology}
respectively).

\begin{definition}
	A general topological space $X$ is \emph{normal} if for
	any two disjoint closed subsets $A,B \subseteq X$, there are open sets
	$U,V \subseteq X$ such that $U \cap V = \emptyset$, $A \subseteq U$, and 
	$B \subseteq V$.
\end{definition}

There is a very useful characterization of normal spaces:

\begin{thm}[Tietze Extension Theorem]
	\label{thm:tietze}
	$X$ is normal if and only if any continuous map $A \to \R$, $A \subseteq
	X$ closed, can be extended to a map $X \to \R$.
\end{thm}

This theorem easily allows us to deduce a few properties of and recognition
criteria for normal spaces:

\begin{lemma}~
	\label{lem:normal}
	\begin{enumerate}
		\item Metrizable spaces and compact Hausdorff spaces are normal.
		\item Closed subspaces of normal spaces are normal.
		\item If $A \subseteq B$ is a closed inclusion, $B$ and $X$ are normal
			spaces, then for any pushout diagram
			\[
				\pushout{A}{X}{B}{Y}{}{}{}{}
			\]
			the space $Y$ is normal.
		\item If $\lambda$ is an ordinal and $(X_\beta)_{\beta < \lambda}$ is
			a $\lambda$-sequence of normal spaces along closed inclusions,
			then the space~$\colim_{\beta < \lambda} X_\beta$ is normal.
		\item Normal $T_1$ spaces are Hausdorff.
		\item Retracts of normal $T_1$ spaces are normal.
		\item Quillen cofibrant topological spaces are normal.
	\end{enumerate}
\end{lemma}
\begin{proof}
	(i) is classical, and (ii)--(iv) can easily be shown using the Tietze Extension
	Theorem. 
	
	For example, let us show (iv):
	Given $(X_\beta)_{\beta < \lambda}$, a closed subspace $A$ of the space
	$\colim_{\beta < \lambda} X_\beta =: X$, and a continuous map
	$f\colon A \to \R$, we wish to find an extension $F\colon X \to \R$ of $f$.
	We start a transfinite recursion in degree $0$. As $A \cap X_0$ is closed
	and $X_0$ is normal, we may find an extension $F_0\colon X_0 \to \R$ of
	$f_{|A \cap X_0}$.  Assume that $\beta = \gamma + 1$ is a successor ordinal and that
	there is an extension $F_\gamma\colon X_\gamma \to \R$ of
	$F_{|A \cap X_\gamma}$.
	Then $X_\gamma \cup (A \cap X_\beta) \subseteq X_\beta$ is closed (because
	the map $X_\gamma \to X_\beta$ has closed image), and the function
	$F_\gamma \cup f_{|A \cap X_\beta}\colon X_\gamma \cup (A \cap X_\beta)
	\to \R$ is continuous (because $X_\gamma \to X_\beta$ is a
	homeomorphism onto its image). Hence, by the normality of $X_\beta$, we
	may extend this function to $F_\beta\colon X_\beta \to \R$. If $\beta$ is a
	limit ordinal and we have already chosen compatible extensions
	$F_\gamma\colon X_\gamma \to \R$, then
	$X_\beta = \colim_{\gamma < \beta} X_\gamma$ and we extend
	$f_{A \cap X_\beta}$ by $\colim_{\gamma < \beta} F_\gamma\colon X_\beta
	\to \R$. Finally, taking
	$F := \colim_{\beta < \lambda} F_\lambda\colon X \to \R$ yields the
	desired extension of $f$.

	Part (v) is immediate because points are closed in $T_1$ spaces. For part (vi), if
	$A \overset{i}\to X \overset{p}{\to} A$ is a retract diagram,
	then $i$ is automatically a homeomorphism onto its image. As $X$ is
	Hausdorff, its retract subspace $i(A)$ is closed. Therefore, it is normal
	by (ii). The last part follows from (i), (iii), (iv), (vi), and
	the characterization of cofibrant objects in cofibrantly generated model
	categories, see~\cite[Lemma~A.4]{koerschgen_dwyer-kan_2017}.
\end{proof}

\begin{lemma}
	\label{lem:egnormal}
	Let $G$ be a compact Lie group. Then $EG$ is normal and Hausdorff.
\end{lemma}
\begin{proof}
	The space $EG$ is the geometric realization of the simplicial topological
	space $EG_\bullet$ which is given by $EG_n = G^{n+1}$.  As for every
	geometric realization, $EG$ is the sequential colimit of its skeleta
	$\sk_n EG$, which sit in pushouts
	\[
		\pushout
			{L_n EG_\bullet \times \Delta^n \cup_{L_n EG_\bullet \times \del \Delta^n}
				EG_n \times \del \Delta^n}
			{\sk_{n-1} EG}{EG_n \times \Delta^n}{\sk_n EG}{}{}{}{}
	\]
	where $L_n EG_\bullet \to EG_n$ is the $n$-th latching map. Since
	$EG_\bullet$ is proper, cf.~the Proof of Proposition~\ref{prop:lhgwe}, the
	latching maps are closed Hurewicz cofibrations. The same is true for the
	vertical pushout-product morphism on the left hand side. In
	particular, these pushout-product morphisms are closed inclusions. As
	$EG_n \times \Delta^n = G^{n+1} \times \Delta^n$ and $\sk_{-1}
	EG_\bullet = \emptyset$ are normal, it follows inductively from
	Lemma~\ref{lem:normal} that the $\sk_n EG$ are normal.

	Moreover, the pushout diagram tells us that the morphisms $\sk_{n-1} EG
	\to \sk_n EG$ are closed inclusions, using
	Lemma~\ref{lem:closedinclusionsclosed}. Hence, $EG = \colim_n \sk_n EG$ is
	normal by Lemma~\ref{lem:normal}. As $EG$ is CGWH, it is $T_1$. Thus, by
	Lemma~\ref{lem:normal}, it is Hausdorff.
\end{proof}

\subsection{Homotopy Quotients of Free \texorpdfstring{$G$}{G}-Spaces}
\label{subsec:freeactionhomquot}

The goal of this subsection is proving the following theorem:

\begin{thm}
	\label{thm:freeactionT2space}
	Let $G$ be a compact Lie group and $X$ a compactly generated Hausdorff space,
	endowed with a free right $G$-action. Then the canonical map $X \times_G
	EG \to X/G$ from the homotopy quotient to the quotient is a weak
	equivalence.
\end{thm}

The crucial point in the proof of this theorem is gaining homotopical control
over the quotient maps $X \times EG \to X \times_G EG$ and $X \to X/G$. Let us
begin by citing several facts about $G$-spaces. 

\begin{thm}
	\label{thm:gspacefacts}
	Let $G$ be a compact Lie group acting on a general topological
	space $X$. Denote by $q\colon X \to X/G$ the
	quotient map. Then the following hold:
	\begin{enumerate}
		\item For any closed subset $A \subset X$, the set $A G = \{ag\sth a
				\in A, g \in G\}$ is closed in $X$.
		\item The quotient map $q$ is closed.
		\item Let $A$ be any $G$-invariant subset of $X$. Then the
			birestriction $A \to q(A)$ of $q$ agrees with the quotient map
			$A \to A/G$, i.e., the subspace topology on $q(A)$ is the
			quotient topology of $A/G$.
		\item If $X$ is a CGWH space, then so is $X/G$. Put differently, the
			quotient computed in general topological spaces agrees with the
			quotient taken in CGWH spaces.
		\item If $X$ is Hausdorff, then the quotient map $q$ is proper.
		\item If $X$ is Hausdorff, then so is $X/G$.
		\item Assume that the following two conditions are
			satisfied:
			\begin{enumerate}
				\item For some fixed $H \leq G$, the stabilizer subgroups of
					all points in $X$ are conjugate to $H$.
				\item The space $X$ is Tychonoff, i.e., completely regular and
					Hausdorff.
			\end{enumerate}
			Then $q$ is a fiber bundle with typical fiber $H\backslash G$.
	\end{enumerate}
\end{thm}
\begin{proof}
	Statement (i) is a special case of~\cite[Proposition
	I.3.1.(iii)]{tom_dieck_transformation_1987}, and the second assertion (ii)
	follows from (i) by the definition of the quotient topology. An
	alternative argument can be found in~\cite[Proposition
	B.13.(ii)]{schwede_global_2017}.  Statements (iv) and (iii) are
	precisely~\cite[Proposition B.13, (i) and (iii)]{schwede_global_2017}.

	The crucial step in proving (v) is ensuring that the fibers
	$q^{-1}(q(x))$ are compact. This is rendered possible by the Hausdorff
	property because it implies that, for any $x \in X$, the canonical map
	$\Stab(x) \backslash G \to xG = q^{-1}(q(x))$ is a homeomorphism. A
	full argument can be found in~\cite[Section
	I.3]{tom_dieck_transformation_1987} or~\cite[Section
	I.3]{bredon_introduction_1972}. These two sources also provide a reference
	for (vi). Note that~\cite{bredon_introduction_1972} exclusively works in
	the setting of general topological spaces that are Hausdorff.

	Finally, condition (a) is a rephrasing of the requirement that all orbits
	have type $H\backslash G$, therefore \cite[Theorem
	II.5.8]{bredon_introduction_1972} applies verbatim.
\end{proof}

However, the point-set property of being Tychonoff is very inconvenient from a
homotopy theorist's point of view because one does not seem to have any grasp
on it under limits. Complete regularity means that closed sets can be
separated from points by Urysohn functions. Most limits in CGWH spaces
necessitate usage of the $k$-ification functor, which potentially introduces
more closed sets and, therefore, destroys complete regularity.

On the other hand, compact Hausdorff spaces are automatically Tychonoff, and
the topology on a CGWH space is generated by its compact Hausdorff subspaces.
This will be the main argument in the proof of the following theorem.

\begin{thm}
	\label{thm:quotientmapfibration}	
	Let $G$ be a compact Lie group acting on a compactly generated Hausdorff
	space $X$.  Assume that for some fixed $H \leq G$, all stabilizers
	subgroups of points in $X$ are conjugate to $H$. Then the quotient map
	$q\colon X \to X / G$ is a fibration.
\end{thm}
\begin{proof}
	The topology on $X$ is generated by its compact subspaces, which are
	automatically Hausdorff, in the sense that the canonical map
	\[
		\colim_{K \subseteq X \text{compact}} K \to X
	\]
	is a homeomorphism. For any compact set $K \subseteq X$, there is a
	$G$-invariant, compact superset $KG$. In
	particular, the $G$-invariant, compact sets form a cofinal subcategory of
	the compact sets, and the natural map
	\[
		\colim_{\substack{K \subseteq X \text{compact }\\
			\text{and }G\text{-invariant}}} K \to X
	\]
	is a homeomorphism. This colimit commutes with taking $G$-quotients.
	Hence,
	\[
		X/G = 
		\colim_{\substack{K \subseteq X \text{compact }\\
			\text{and }G\text{-invariant}}} (K/G).
	\]
	
	For any $G$-invariant, compact $K \subseteq X$, the space $K$ is compact
	Hausdorff which implies normal Hausdorff and, consequently, the Tychonoff
	property. The stabilizer subgroups of the $G$-action on $K$ are conjugate
	to a fixed $H$, so we may apply Theorem~\ref{thm:gspacefacts}.(vii) and
	deduce that the quotient map $K \to K/G$ is a fiber bundle and, therefore,
	a fibration. Moreover, note that this quotient map is exactly the birestriction
	$q_{|K}$ of $q\colon X \to X/G$ to $K$ and its image $q(K) = K/G$,
	see Theorem~\ref{thm:gspacefacts}.(iv). 

	Having made the necessary preparations, let us consider a
	generating acyclic cofibration of topological spaces
	$i\colon A \hookrightarrow  B$ and a lifting problem
	\[
		\begin{tikzcd}
			A \rar{f} \dar{i} & X \dar{q} \\
			B \rar{g} & X/ G
		\end{tikzcd}
	\]
	As $B$ is compact, its image $g(B) \subseteq X/G$ is
	compact. Since $q$ is proper by Theorem~\ref{thm:gspacefacts}.(v), the
	$G$-invariant subset $K := q^{-1} (g(B))$ is compact, too.
	Also, $g(B) = q
	(q^{-1} (g(B))) = q(K) = K/G$. By corestriction, our lifting problem
	reduces to the left-hand square of
	\[
		\begin{tikzcd}[column sep=large]
			A \rar{f} \dar{i} & K \dar{q_{|K}} \rar & X \dar{q} \\
			B \rar{g} \urar[dashed] & K/G \rar & X/G
		\end{tikzcd}
	\]
	The middle vertical map $q_{|K}$ is a fibration by our previous
	considerations, and we may find a lift as indicated. Postcomposing with
	the inclusion $K \hookrightarrow X$ immediately solves our original
	lifting problem.
\end{proof}

\begin{rem}
	In the previous proof, the Hausdorff property for $X$ ensures that the
	space $K = q^{-1} (g(B))$ is compact Hausdorff. Compactness follows from
	the properness of $q$, which necessitates the Hausdorff property, and
	the Hausdorff property for $K$ itself is inherited from $X$. Despite
	several efforts, we were unable to eliminate the Hausdorff requirement for
	$X$. However, we believe that the statement in its current form may prove
	useful for other applications as compact Lie group actions on Hausdorff
	spaces are ubiquitous.
\end{rem}

Before we proceed to the proof of Theorem~\ref{thm:freeactionT2space}, let us
record some preparations which will help us to establish isomorphisms on
$\pi_0$ and $\pi_1$.

\begin{prop}
	\label{prop:quotientpi0}
	Let $X$ and $G$ be as before. The $G$-action on $X$ induces a
	$\pi_0 (G)$-action on $\pi_0 (X)$, and there is a dashed arrow making the
	following triangle commute: 
	\[
		\begin{tikzcd}
			\pi_0 (X) \rar{\pi_0(q)} \dar{\rho} &
			\pi_0 (X/G) \\
			\pi_0 (X) / \pi_0 (G)\urar[dashed,swap]{\overline{\pi_0(q)}} 
		\end{tikzcd}
	\]
	This arrow is an isomorphism and natural with respect to morphisms of $G$-spaces.
\end{prop}
\begin{proof}
	The $\pi_0 (G)$-action on $\pi_0 (X)$ is well-defined because $\pi_0$
	commutes with products, and $G$-morphisms $X \to Y$ give rise to morphisms
	$\pi_0 (X) \to \pi_0 (Y)$ of $\pi_0(G)$-sets. If it exists, the arrow
	$\overline{\pi_0(q)}$ is unique which implies naturality.
	
	Denote path-components by $[\cdot]$. In order to show that
	$\overline{\pi_0(q)}$ exists, pick $[x] \in \pi_0 (X)$. Then for all
	$[g] \in \pi_0 (G)$, we have $\pi_0 (q) ([x]) = [q(x)] = [q(xg)] =
	\pi_0 (q) ([xg]) = \pi_0 (q) ([x] \cdot [g])$ where $\cdot$ denotes
	the $\pi_0 (G)$-action on $\pi_0 (X)$. Thus, $\pi_0(q)$ factors through
	the quotient map $\rho$.

	If $X = \emptyset$, then both domain and codomain of
	$\overline{\pi_0(q)}$ are one-point sets and the map is an isomorphism.
	Otherwise, the quotient map $q\colon X \to X/G$ is surjective and induces a
	surjection on $\pi_0$. Consequently, the map $\overline{\pi_0(q)}$ is
	surjective as well, and it remains to show that the map is injective.

	Assume that for $x,y \in X$, we have $\overline{\pi_0(q)}(\rho([x])) =
	\overline{\pi_0(q)}(\rho([y]))$, i.e., $\pi_0(q) ([x]) =
	\pi_0(q)([y])$, which can be rephrased as $[q(x)] = [q(y)]$. This
	means that there is a path $\gamma\colon [0;1] \to X/G$ from $q(x)$ to
	$q(y)$. As the quotient map $X \to X/G$ is a fibration, we can lift
	$\gamma$ to a path $\tilde{\gamma}\colon [0;1] \to X$ from
	$\tilde{\gamma}(0) = x$ to some point $\tilde{\gamma}(1)$ with
	$q(\tilde{\gamma}(1)) = q(y)$.
	This implies that there is $g \in G$ with $\tilde{\gamma}(1) = y g$. In
	particular, we have
	\[
		\rho([x]) = \rho([\tilde{\gamma}(1)]) = \rho([yg]) = \rho([y] \cdot
		[g]) = \rho([y])
		\in \pi_0 (X) / \pi_0 (G),
	\]
	and we deduce that $\overline{\pi_0(q)}$ is injective.
\end{proof}

\begin{definition}
	\label{def:fibergroup}
	Let $X$ and $G$ be as before. Assume additionally that $X \neq \emptyset$
	and that the $G$-action on $X$ is free. Picking $x \in X$, denote by
	$\fib_{q(x)} (q)$ the fiber of $q\colon X \to X/G$ over $q(x) \in X/G$.
	We equip this fiber with a group structure using the homeomorphism
	\[
		G \overset{\cong}{\to} \fib_{q(x)}(q),\ g \mapsto x g^{-1}.
	\]
\end{definition}
\begin{rem}
	The assignment $g \mapsto x g$ is a homeomorphism as well and could have
	been used to define the group structure. However, the convention described
	above ensures that the group structure is compatible with the standard
	conventions for homotopy groups. Moreover, note that this group structure
	is natural with respect to $G$-equivariant pointed maps of pointed free
	$G$-spaces.
\end{rem}

\begin{prop}
	\label{prop:delgrouphom}
	Let $X$, $x$, and $G$ be as in Definition~\ref{def:fibergroup}. The
	boundary map
	\[
		\pi_1(X/G,q(x)) \overset{\del}{\to} \pi_0 (\fib_{q(x)}(q))
	\]
	associated to the fibration $q\colon X \to X/G$ is a group homomorphism with
	respect to the group structure which is induced by the group structure
	from Definition~\ref{def:fibergroup} under the functor $\pi_0$.
\end{prop}
\begin{proof}
	We interpret $\pi_1 (X/G,q(x))$ as classes of maps $(D^1,S^0) \to
	(X/G,q(x))$ of pairs of spaces
	and employ the notation $[\cdot]$ for classes in
	$\pi_0$ or $\pi_1$. Let $\alpha_1$ be such a map. As $q$ is a
	fibration, we may find a lift $\widetilde{\alpha_1}\colon D^1 \to X$ such
	that $q \circ \widetilde{\alpha_1} = \alpha_1$ and $\widetilde{\alpha_1}
	(0) = x$.  Then $\del[\alpha_1] = [\widetilde{\alpha_1}(1)] \in
	\fib_{q(x)} (q)$. There is a unique $g \in G$ such that
	$\widetilde{\alpha_1}(1) = x g^{-1}$.

	Given another $\alpha_2\colon (D^1,S^0) \to (X/G,q(x))$, choose a lift
	$\widetilde{\alpha_2}$ as before such that $\del[\alpha_2] =
	[\widetilde{\alpha_2}(1)]$. Let $h \in G$ be the unique element with
	$\widetilde{\alpha_2}(1) = x h^{-1}$. We wish to show that $\del([\alpha_1]
	\cdot [\alpha_2])$ agrees with
	$\del([\alpha_1])\,\cdot\,\del([\alpha_2]) =
	[x g^{-1}] \cdot [x h^{-1}]$ which is $[x (gh)^{-1}] = [x h^{-1} g^{-1}]$ by
	definition of the group structure on $\pi_0(\fib_{q(x)}(q))$.
	
	The product $[\alpha_1] \cdot [\alpha_2]$ in $\pi_1 (X,x)$ is represented
	by the concatenation $\alpha_1 \ast \alpha_2$ with the standard convention
	$(\alpha_1 \ast \alpha_2)(0) = \alpha_1(0)$, $(\alpha_1 \ast \alpha_2)(1)
	= \alpha_2(1)$. Denote by $\widetilde{\alpha_2} \cdot g^{-1}$ the path $t
	\mapsto \widetilde{\alpha_2} (t) \cdot g^{-1}$ from 
	$x g^{-1}$ to $x h^{-1} g^{-1}$.
	Then $\widetilde{\alpha_1}$ and $\widetilde{\alpha_2} \cdot g^{-1}$ are
	concatenable and their concatenation $\gamma := \widetilde{\alpha_1} \ast
	(\widetilde{\alpha_2} \cdot g^{-1})$ satisfies $q \circ \gamma =
	\alpha_1 \ast \alpha_2$ and $\gamma(0) = x$. Hence, $\del([\alpha_1] \cdot
	[\alpha_2]) = \del ([\alpha_1 \ast \alpha_2)] = [\gamma(1)] = 
	[x h^{-1} g^{-1}]$ as desired.
\end{proof}

\begin{proof}[{Proof of Theorem~\ref{thm:freeactionT2space}}]
	In the trivial case $X=\emptyset$, the statement holds. Let us exclude
	this case from the following considerations.

	Observe that $X$ and $EG$ are Hausdorff by assumption and by
	Lemma~\ref{lem:egnormal}, respectively. Therefore, their product $X
	\times_0 EG$ in the category of general topological spaces is Hausdorff.
	As all open sets in $X \times_0 EG$ are also open in its $k$-ification $X
	\times EG = k (X \times_0 EG)$, we deduce that $X \times EG$ is Hausdorff
	as well.

	Denote the projection $X \times EG \to X$ by $\pr$ and pick $(x,e) \in X
	\times EG$. We obtain a commutative diagram of fiber
	sequences:
	\[
		\begin{tikzcd}
			\fib_{q_1(x,e)} (q_1) \rar \dar[swap]{\fib(\pr,\pr/G)} &
			X \times EG \rar{q_1} \dar{\pr} & 
			X \times_G EG \dar{\pr /G} \\
			\fib_{q_2(x)} (q_2) \rar & X \rar{q_2} & X /G
		\end{tikzcd}
	\]
	Using homeomorphisms as in Definition~\ref{def:fibergroup}, we obtain
	a commutative triangle
	\[
		\begin{tikzcd}
			& \fib_{q_1(x,e)}(q_1) = (x,e)G \arrow[dd,"{\fib(\pr,\pr/G)}"] \\
			G/\{1\} \cong G \urar{\cong} \drar[swap]{\cong} \\
			& \fib_{q_2(x)}(q_2) = xG
		\end{tikzcd}
	\]
	Thus, $\fib(\pr,\pr/G)$ is a homeomorphism. Since $EG$ is contractible,
	$\pr\colon X \times EG \to X$ is a weak equivalence.  
	By Theorem~\ref{thm:quotientmapfibration}, the quotient maps $q_1$ and $q_2$
	are fibrations. In the associated long exact sequences on $\pi_\ast$,
	the terms $\pi_0 (\fib_{q_1(x,e)}(q_1))$ and $\pi_0 (\fib_{q_2(x)}(q_2))$
	inherit group structures, and the appropriate boundary maps are group
	homomorphisms, see Definition~\ref{def:fibergroup} and
	Proposition~\ref{prop:delgrouphom}. This is sufficient to deduce that
	$\pr/G$ induces an isomorphism on $\pi_n$ for $n \geq 1$ and all
	basepoints by the four lemmas.

	Finally, the weak equivalence $\pr\colon X \times EG \to X$ is
	$G$-equivariant and induces an isomorphism of $\pi_0 (G)$-sets on $\pi_0$.
	Thus, it yields an isomorphism $\pi_0(\pr)/\pi_0(G)$ from $\pi_0 (X \times
	EG)/\pi_0 (G)$ to  $\pi_0 (X)/\pi_0 (G)$. The diagram
	\[
		\begin{tikzcd}[column sep=5.4em,row sep=large]
			\pi_0 (X \times EG)/\pi_0 (G)
				\arrow[r,"\pi_0 (\pr) / \pi_0 (G)","\cong"']
				\arrow[d,"\overline{\pi_0(q_1)}"',"\cong"] &
			\pi_0 (X) / \pi_0 (G)
				\arrow[d,"\overline{\pi_0(q_2)}","\cong"'] \\
				\pi_0 (X \times_G EG) \rar[swap]{\pi_0(\pr /G)} &
			\pi_0 (X/G)
		\end{tikzcd}
	\]
	commutes, and the vertical arrows are isomorphisms, see
	Proposition~\ref{prop:quotientpi0}. Therefore, the map $\pr /G$ is an
	isomorphism on $\pi_0$, too, and it is a weak equivalence.
\end{proof}

\subsection{The Topology of \texorpdfstring{$\et(H,G)$}{\etPDF(H,G)}}
\label{subsec:ethgtop}

In order to apply Theorem~\ref{thm:freeactionT2space} to $\et(H,G)$, which is
necessary for the proof of Proposition~\ref{prop:khgwe}, we need to verify
that this is a Hausdorff space. Let us start by recalling the topology on
$E(H,G)$.

\begin{obs}
	\label{obs:Ltopology}
	The underlying space of the $(H \times G)$-space $E(H,G)$ is the space of
	linear isometries $\L(\ri,\ri)$, whose topology we will define now.
	First, given finite-dimensional real inner product spaces $V,W$, the space
	$\L(V,W)$ of linear isometric embeddings is topologized by choosing an
	orthonormal basis $(v_1,\ldots,v_k)$ of $V$ and using the bijection
	\[
		\begin{array}[]{ccc}
			\L(V,W) & \to & \sti_k (W), \\
			f & \mapsto & (f(v_1),\ldots,f(v_k))
		\end{array}
	\]
	to the Stiefel manifold of $k$-frames in $W$. This Stiefel manifold is a
	subset of $W^k$ and endowed with the subspace topology.

	Having established the topology on $\L(V,W)$ for $V$ and $W$
	finite-dimensional, we move on to the case where the second argument is of
	countably infinite dimension. For such a $\cW$, we define
	\[
		\L(V,\cW) := \colim_{W \in s(\cW)} \L(V,W)
	\]
	where $s(\cW)$ is the poset of finite-dimensional subspaces of $\cW$. For
	$W \leq W^\prime \in s(\cW)$, the induced map $\L(V,W) \to \L(V,W^\prime)$
	is a closed inclusion.  
	Furthermore, If $W_0 \subseteq W_1 \subseteq \dots \subseteq \cW$ is a
	strictly increasing sequence of finite-dimensional subspaces such that
	$\cup_l W_l = \cW$, it is cofinal, and we have
	\[
		\L(V,\cW) = \colim_{l} \L(V,W_l).
	\]
	As this colimit is taken along closed inclusions, it agrees with the
	colimit computed in general topological spaces
	(Lemma~\ref{lem:gentopspacessomecolimitsagree}).

	It remains to upgrade the first variable to vector spaces of countably
	infinite dimension. For such a vector space $\cV$, we set
	\[
		\L(\cV,\cW) = \lim_{V \in s(\cV)} \L(V,\cW)
	\]	
	where the limit is taken in the category of CGWH spaces. Again, for an
	exhaustive sequence $V_0 \subseteq V_1 \subseteq \dots$ of finite dimensional
	subspaces of $\cV$, we have
	\[
		\L(\cV,\cW) = \lim_l \L(V_l,\cW).
	\]
	By~\cite[Appendix~A]{schwede_orbispaces_2017}, this is the same
	topology as the subspace topology inherited from the inclusion
	$\L(\cV,\cW) \subseteq \map(\cV,\cW)$ into the space of all continuous
	maps.
\end{obs}

\begin{lemma}
	\label{lem:lririT2}
	The space $\L(\ri,\ri)$ is Hausdorff.
\end{lemma}
\begin{proof}
	First, the spaces $\L(\R^n,\R^m)$ are normal as they are metrizable.
	Therefore, $\L(\R^n,\ri)$, which is defined as $\colim_m \L(\R^n,\R^m)$,
	is a sequential colimit of normal spaces along closed inclusions. Thus,
	$\L(\R^n,\ri)$ is normal by Lemma~\ref{lem:normal}. The space
	$\L(\R^n,\ri)$ is $T_1$, too, so it is Hausdorff.

	The limit $\L(\ri,\ri) = \lim_n \L(\R^n,\ri)$ can be modeled as a
	(closed) subspace of the infinite product $\prod_n \L(\R^n,\ri)$ of
	Hausdorff spaces. As subspaces of Hausdorff spaces are Hausdorff, it
	suffices to show that $\prod_n \L(\R^n,\ri)$ is Hausdorff. This is true
	for the product computed in general topological spaces. The product in the
	category of CGWH spaces can then be obtained by $k$-ification. However,
	this procedure does not destroy Hausdorffness because $k$-ification
	preserves open sets. We conclude that $\prod_n \L(\R^n,\ri)$ and its
	subspace $\L(\ri,\ri)$ are Hausdorff.
\end{proof}

We are now ready to obtain a few important facts about the topology on
$\et(H,G)$.
\begin{lemma}
	\label{lem:ethgtopology}
	The two product topologies $\map(H,G) \times E(H,G)$ and $\map(H,G) \times_0 E(H,G)$
	agree, and they are Hausdorff. The subspace $\et(H,G)$ is compactly generated
	and Hausdorff. In particular, it is CGWH.
\end{lemma}
\begin{proof}
	The space $\map(H,G)$ is metrizable and hence locally compact Hausdorff.
	Thus, the product $\map(H,G) \times E(H,G)$ agrees with the product
	$\map(H,G) \times_0 E(H,G)$ taken in general topological spaces.

	For the second statement, first observe that $\map(H,G)$ is Hausdorff
	by the discussion above and that the space $E(H,G)$ is Hausdorff by
	Lemma~\ref{lem:lririT2}. Therefore, $\map(H,G) \times E(H,G)$ is Hausdorff
	and so is its subspace $\et(H,G)$. This implies the weak Hausdorff
	property of course.

	In order to prove that $\et(H,G)$ is compactly generated, recall that 
	$\et(H,G)$ is a fiber bundle over $\map(H,G)$
	by Proposition~\ref{prop:ptildefiberbundle}. In particular, the proof
	shows that each point of $\et(H,G)$ has a neighborhood that is
	homeomorphic to $U \times E(H,G)^\alpha$ for $U \subseteq \map(H,G)$ open
	and $\alpha \in \map(H,G)$.

	According to our convention, the product $U \times E(H,G)^\alpha$ is
	computed in the category of CGWH spaces, and therefore automatically
	compactly generated. As $U$ is locally compact Hausdorff, we actually have
	$U \times E(H,G)^\alpha = U \times_0 E(H,G)^\alpha$.

	In any case, the space $\et(H,G)$ is locally compactly generated which
	implies that it is compactly generated, concluding the proof.
\end{proof}

\bibliographystyle{abbrv}
\bibliography{orb_comparison}

\vspace{2em}
\noindent
Mathematisches Institut der Universität Bonn\\
Endenicher Allee 60, 53115 Bonn, Germany

\vspace{1em}
\noindent
E-mail address: \href{mailto:math@koerschgen.name}{math@koerschgen.name}

\end{document}